\documentclass[11pt]{article} 
\usepackage{etoolbox}
\usepackage{setspace}
\usepackage{amsmath,amsthm,amsfonts,amssymb,amscd,mathrsfs}
\usepackage{tikz}
\usepackage{tikz-cd}
\usepackage{float}
\usepackage{subcaption}
\usepackage[latin1]{inputenc}
\usepackage[english]{babel}
\IfFileExists{algorithm.sty}{\usepackage{algorithm}}{} 
\usepackage{adjustbox}
\usepackage{cite}
\usepackage{graphicx}
\usepackage{colortbl,dcolumn}
\usepackage{marvosym}
\IfFileExists{ifsym.sty}{\usepackage{ifsym}}{} 
\usepackage{paralist}
\usepackage{xcolor}
\usepackage{array}
\usepackage{geometry}
\usepackage{enumitem}

\definecolor{myPurple}{RGB}{109,90,207}
\definecolor{myGreen}{RGB}{4,130,67}
\definecolor{myGold}{RGB}{253,177,71}
\definecolor{myBurgundy}{RGB}{63,1,44}
\definecolor{myTeal}{RGB}{8,119,127}
\definecolor{myCream}{RGB}{255,216,177}
\definecolor{myOrange}{RGB}{225,119,1}
\definecolor{mydeepgreen}{RGB}{3,100,50}

\usetikzlibrary{shapes.geometric,arrows.meta,calc,decorations.markings,bending,arrows}

\pgfmathsetmacro{\wNear}{2.0pt}
\pgfmathsetmacro{\wMid}{1.3pt}
\pgfmathsetmacro{\wFar}{0.7pt}
\pgfmathsetmacro{\bend}{6mm}

\makeatletter
\patchcmd{\thebibliography}
{\list}
{\small\setstretch{1.2}\list}
{}{}
\makeatother

\allowdisplaybreaks
\newtheorem{lemma}{\bf Lemma}[section]
\newtheorem{theorem}{\bf Theorem}[section]
\newtheorem{proposition}{\bf Proposition}[section]
\newtheorem{corollary}{\bf Corollary}[section]
\newtheorem{definition}{\bf Definition}[section]
\newtheorem{remark}{\bf Remark}[section]

\numberwithin{equation}{section}

\usepackage[
colorlinks,
citecolor=blue,
linkcolor=blue]{hyperref}

\begin{document}
	\title{{\sl Optimal rates of uniform convergence for weighted Birkhoff averages via almost all rotations}}
	\author{Zhicheng Tong\thanks{School of Mathematics, Jilin University, Changchun 130012, P. R.  China. Email: \url{tongzc25@jlu.edu.cn}} 
		\and 
		Yong Li\thanks{Corresponding author. School of Mathematics, Jilin University, Changchun 130012, P. R.  China;   Center for Mathematics and Interdisciplinary Sciences,  Northeast  Normal  University, Changchun 130024, P. R. China. Email: \url{liyong@jlu.edu.cn}}
	}
	
	\date{}
	\maketitle

	\begin{abstract}
		In this paper, we investigate weighted Birkhoff averages for toral translations associated with compactly supported weighting functions. By introducing several new analytical techniques, we establish optimal uniform convergence rates for almost all rotations and specific (or even all) initial points. Unlike the $\mathcal{O}(N^{-1})$ rate best achieved in classical ergodic theory, we show that these weighted averages exhibit polynomial or even exponential convergence. We establish the optimality of these convergence rates in multiple aspects, particularly concerning regularity indices across four distinct cases: finite differentiability, the $C^\infty$ class, logarithmic $C^\infty$ classes, and Gevrey classes. Our results demonstrate that the regularity of the observable essentially dictates the convergence rate; furthermore, we prove that no admissible choice of weighting function can, in general, overcome the lower bounds imposed by this regularity. In contrast to the generically slow convergence of standard time averages, this work provides an optimal and nearly complete characterization of rapid convergence for weighted Birkhoff averages.\vspace{2mm}
	\end{abstract}
	
	\tableofcontents
	
	\section{Introduction}
	\setcounter{footnote}{0}
	\renewcommand{\thefootnote}{{\arabic{footnote}}}
	
	The ergodic theorem due to Birkhoff and von Neumann is recognized as a fundamental principle in the sciences, finding extensive applications across various disciplines such as statistical mechanics and celestial mechanics. This celebrated theorem asserts that, under the general ergodic hypothesis, the time average (also termed the Birkhoff average) of an observable possessing suitable regularity tends to its spatial average. More precisely, given a measure-preserving transformation $T$ on a topological space $X$ equipped with a probability measure $\mu$ for which $T$ is ergodic, the Birkhoff average of an observable $f$ defined on $X$ is given by
	$$
	\mathrm{B}_N(f)(x) := \frac{1}{N}\sum_{n = 0}^{N - 1} f(T^n x), \quad x \in X.
	$$
	As the number of iterations $N$ increases, $\mathrm{B}_N(f)(x)$ converges to the spatial average $\int_X f d\mu$ in an appropriate sense; for instance, in the $L^2$-norm when $f \in L^2(X, \mu)$ (von Neumann's ergodic theorem), and almost everywhere when $f \in L^1(X, \mu)$ (Birkhoff's ergodic theorem).
	
	In general settings, however, the convergence of these Birkhoff averages is exceedingly slow. It was initially shown in \cite{Kre78} that one can construct a continuous observable for which the Birkhoff average exhibits \textit{arbitrarily slow convergence}. Similar yet distinct statements were proved in \cite{dR79,Pod22,Ryz23,Ryz25,PR26}, among others. These fundamental observations are commonly referred to as \textit{``the absence of estimates for rates of convergence in the Birkhoff ergodic theorem''}. Therefore, to obtain explicit rates of convergence for Birkhoff averages, one must restrict the study to specific dynamical systems. Moreover, establishing quantitative rates is a highly challenging yet significant problem that remains an active area of research. The aforementioned results illustrate the pathological behavior of the Birkhoff average in special counterexamples, whereas for general observables, the convergence rate is still fundamentally limited. It has been well known since \cite{Kac96} that for non-constant observables, the convergence rate of the Birkhoff average is \textit{at best} $\mathcal{O}(N^{-1})$, i.e., inversely proportional to the number of iterations; see also \cite{KP16,Pod22,Pod24,PR26} on this aspect. Furthermore, it is important to note that even this optimal rate is achieved only under highly restrictive assumptions \cite{KH95,Kat03,DY18,KLM21}.
	
	To make the preceding results more transparent to the reader, let us consider the case of primary interest in ergodic theory: toral translations. Although such systems have been studied extensively due to their profound connections to number theory, integrable systems, and geometry, a surprising number of natural questions remain open, as emphasized in \cite{DF15}. We now focus on the Birkhoff averages associated with toral translations. Given an irrational rotation on the $d$-torus, defined by
	\begin{equation}\label{xuanzhuan}
		\mathscr{T}_\rho^s(\theta) := \theta + s\rho \pmod{\mathbb{Z}^d}, \quad \theta \in \mathbb{T}^d, \quad d \in \mathbb{N}^+, \quad s \in \mathbb{R},
	\end{equation}
	where the rotation vector $\rho \in \mathbb{R}^d$ satisfies a nonresonance condition (such as a Diophantine condition, cf. \cite{Her79}), the convergence rate of the Birkhoff average
	\begin{equation}\label{OPTxz}
		\frac{1}{N}\sum_{n = 0}^{N - 1} f\left( \mathscr{T}_\rho^n(\theta) \right)
	\end{equation}
	for an observable $f$ with a certain regularity, as well as its continuous counterpart, has attracted considerable interest. Counterexamples on $\mathbb{T}^d$ with $d \geqslant 2$ via extremely Liouvillean rotations\footnote{A rotation vector is said to be Liouville if it is not Diophantine. Strongly Liouville rotation vectors are characterized by being extremely well approximated by rational vectors; consequently, such rotation vectors are non-universal in the measure-theoretic sense.} were constructed in \cite{Yoc80,Yoc95}, yielding \textit{arbitrarily slow convergence} for \eqref{OPTxz} with analytic observables, in analogy with the result in \cite{Kre78}. Even leaving Liouvillean rotations aside, universal ones, which constitute a set of full Lebesgue measure, still \textit{cannot} guarantee rapid convergence for analytic observables. More precisely, in the simplest case on $\mathbb{T}$, the cohomological equation
	$$
	f(\theta) - \int_{\mathbb{T}} f(x) dx = \widetilde{f}\left( \mathscr{T}_\rho(\theta) \right) - \widetilde{f}(\theta), \quad \theta \in \mathbb{T}
	$$
	admits a non-constant continuous solution $\widetilde{f} \colon \mathbb{T} \to \mathbb{R}$ for almost all $\rho \in \mathbb{R}$, provided that the observable $f$ is non-constant and analytic on $\mathbb{T}$\footnote{Indeed, $C^{1+\alpha}(\mathbb{T})$ regularity for any $\alpha > 0$ is sufficient (cf. \cite{Kat04}).}. Hence, a straightforward calculation of \eqref{OPTxz} yields
	$$
	\frac{1}{N}\sum_{n = 0}^{N - 1} f\left( \mathscr{T}_\rho^n(\theta) \right) = \int_{\mathbb{T}} f(x) dx + \frac{\widetilde{f}(\theta + N\rho) - \widetilde{f}(\theta)}{N},
	$$
	which implies a convergence rate of $\mathcal{O}(N^{-1})$; in fact, this rate cannot be improved to $o(N^{-1})$, by density of the irrational rotation orbits and nonconstancy of $\widetilde f$. As one would expect, this rate may deteriorate for observables with lower regularity, such as H\"older continuous functions. In this direction, \cite{KLM21} investigated the upper and lower bounds for the convergence rate of the Birkhoff average \eqref{OPTxz} for almost all irrational rotations, considering observables with $\alpha$-H\"older regularity ($0 < \alpha < 1$). The rate is bounded by $\mathcal{O}(N^{-\alpha}(\log N)^{3\alpha})$ in general, but cannot be faster than $\mathcal{O}(N^{-\alpha})$ for certain counterexamples, rendering it optimal up to a logarithmic factor. Beyond the aforementioned deterministic setting, the random case has also been investigated in \cite{DF14, DF15, DF20}, where slow convergence rates for limit theorems were established.
	
	Beyond the theoretical considerations discussed above, the exceedingly slow convergence of Birkhoff averages in numerical simulations renders practical computations prohibitively difficult, notwithstanding theoretical guarantees of qualitative convergence. In the pursuit of high-precision numerical results, certain computations may require \textit{billions of years} to terminate; we refer to \cite[Section 1.9]{DY18} for further details.
	
	In light of the slow convergence observed from both theoretical and numerical perspectives, accelerating the convergence rate of the Birkhoff average is widely recognized as a challenging problem, motivating the development of several weighted approaches. Specifically, a (non-uniform) weighting scheme for toral translations has been introduced: given a weighting function $w(x)$ satisfying certain special properties, it is sampled uniformly over the time scale and assigned to each term of the standard Birkhoff average \eqref{OPTxz}, yielding the following \textit{weighted Birkhoff average}:
	\begin{equation}\label{lsjq} 
		\frac{1}{A_N}\sum_{n = 0}^{N - 1} w(n/N) f\left( \mathscr{T}_\rho^n(\theta) \right), \quad A_N = \sum_{n = 0}^{N - 1} w(n/N). 
	\end{equation}
	This approach can also be interpreted as a \textit{weighted quasi-Monte Carlo method}; we refer the reader to \cite[Section 1.1.1]{TL26a} for a diagrammatic illustration. In the 1990s, the Hanning window
	\[w_{\mathrm{sin}}(x) := 
	\begin{cases}
		2\sin^2(\pi x), & x \in (0,1), \\
		0, & x \notin (0,1)
	\end{cases}\]
	was employed to accelerate numerical computations for quasi-periodic perturbations of quasi-periodic flows\footnote{This constitutes a continuous analogue of \eqref{lsjq}; despite pursuing distinct objectives, they share similar underlying principles. In recent work, we revisit Laskar's frequency map analysis and establish exponential-type estimates for both the quasi-periodic and almost periodic cases \cite{TL26c}.}, achieving a convergence rate of polynomial order \cite{Las93a,Las93b,Las99,Las05}. This approach proved highly instrumental in celestial mechanics computations at the time and was subsequently termed \emph{frequency map analysis}. Furthermore, a specific exponential weighting function was introduced, exhibiting remarkable asymptotic properties for convergence acceleration \cite[Remark 2]{Las99}; however, its rigorous convergence analysis and numerical implementation for weighted Birkhoff averages were deferred to subsequent studies\footnote{Nevertheless, a rigorous justification for frequency map analysis follows as a direct consequence of \cite[Theorem 1]{Las99}, albeit not explicitly formulated as a theorem therein.}. This weighting function, hereafter referred to as Laskar's function, is given by 
	\begin{equation}\label{weightingfunction}
		w_{\rm Las}(x) := 
		\begin{cases}
			\left( \int_0^1 \exp\left(-s^{-1}(1 - s)^{-1}\right) ds \right)^{-1} \exp\left(-x^{-1}(1 - x)^{-1}\right), & x \in (0,1), \\
			0, & x \notin (0,1).
		\end{cases}
	\end{equation}
	In subsequent applications, its parameterized variant, 
	\begin{equation}\label{bianti}
		w_{\rm Las}^{p,q}(x) := 
		\begin{cases}
			\left( \int_0^1 \exp\left(-s^{-p}(1 - s)^{-q}\right) ds \right)^{-1} \exp\left(-x^{-p}(1 - x)^{-q}\right), & x \in (0,1), \\
			0, & x \notin (0,1)
		\end{cases}
	\end{equation}
	for general $p, q > 0$, has also been extensively employed.
	
	Remarkably, theoretical investigations have demonstrated that this weighted approach significantly accelerates the convergence of the original Birkhoff average \eqref{OPTxz}. The first rigorous theoretical justification for this phenomenon was provided approximately twenty-five years later. Specifically, it was established in \cite{DSSY17,DY18} that, assuming the observable $f$ is of class $C^\ell$ and the rotation vector $\rho$ is Diophantine---where $\ell$, the Diophantine index, and the dimension of the torus satisfy appropriate compatibility conditions---the weighted Birkhoff average \eqref{lsjq} generated by Laskar's function converges uniformly to the spatial average $\int_{\mathbb{T}^d} f(x) dx$ at a finite polynomial rate. Moreover, if $f$ is of class $C^\infty$, this rate can be improved to uniform convergence of arbitrary polynomial order. Analogous conclusions carry over to the continuous setting, as established independently in \cite{DM23} and \cite{TL24a}. Subsequently, by introducing novel techniques and broader frameworks, uniform convergence at arbitrary polynomial and even exponential rates was established in \cite{TL24a} for both finite- and infinite-dimensional toral translations across discrete and continuous dynamical systems. These results rely on specific integrability or smallness conditions that are intimately tied to the regularity of $f$ and the nonresonance properties of $\rho$. In particular, for an analytic observable $f$, the weighted Birkhoff average \eqref{lsjq} associated with Laskar's function \eqref{weightingfunction} (and its continuous counterpart) exhibits \textit{exponential convergence} for almost all $\rho$. More recently, this result was generalized to the \textit{infinite-dimensional} setting in \cite{TL24b}, albeit with a relatively slower convergence rate. Subsequent research has further explored Ces\`aro weighted Birkhoff averages \cite{TL25a}, multiple weighted Birkhoff averages \cite{TL24b}, and weighted averages of decaying waves \cite{TL25b}, in addition to the profound interplay between regularity and nonresonance \cite{Ton26}. By providing \textit{quantitative} refinements to the majority of these prior results, a \textit{comprehensive and user-friendly synthesis} for numerical simulations via the weighting function \eqref{bianti} was developed in \cite{TL26a}; concurrently, the first result on the exponentially accelerated computation of Fourier coefficients was also presented therein. Collectively, all the aforementioned contributions form a robust theoretical foundation for the super convergence observed in practice across various quasi-periodic and almost periodic dynamical systems. In recent years, the volume of literature dedicated to numerical simulations using this weighted approach has expanded rapidly, even outpacing the theoretical advancements discussed above. For representative numerical studies, we refer the reader to \cite{DSSY17,DY18,SM20,MS21,DM23,BM24,CCGd24,RB24,BdJC25,MS25,SM25a,SM25b,RB25,BHS26,BBC26,Bra26,GZE26}; a comprehensive survey can be found in \cite[Section 1.1.2]{TL26a}.
	
	However, in a vein similar to the comments in \cite{DF15}, numerous questions concerning weighted Birkhoff averages for toral translations remain open. In particular, from the standpoints of dynamical systems and number theory, a wealth of intriguing theoretical problems remains open and deserves further investigation. For instance, it is natural to address the fundamental but challenging question of \textit{whether the convergence rate of the weighted Birkhoff average \eqref{lsjq} can be essentially improved by increasing the regularity of the observable or by employing an alternative weighting function}. To this end, as one of the main contributions of this paper, we prove the following informal result, which also makes progress toward the open question recently posed in \cite{PR26}:
	
	\begin{itemize}
		\item Suppose we are given a weighting function satisfying certain properties (see Definition \ref{generalwei}), which in particular encompasses Laskar's function \eqref{weightingfunction} and its parameterized variant \eqref{bianti}. Then, for \textit{almost every} rotation and \textit{specific (or even all)} initial points, the uniform convergence rates of the weighted Birkhoff averages \eqref{lsjq}---whether for observables with finite differentiability, or those in $C^\infty$, logarithmic $C^\infty$, or Gevrey classes---can be \textit{quantitatively} estimated and shown to be \textit{optimal}. Such optimality is demonstrated in twenty-one specific aspects, notably including the regularity indices; see Theorem \ref{OPTT2} and Section \ref{SECOPT} for details. Furthermore, the established lower bounds show that this weighted approach is \textit{unimprovable} in general: no admissible weighting function can provide a faster uniform rate over the prescribed observable classes\footnote{For an explanation, see Item \ref{opitem:h} in Section \ref{SECOPT} and the concluding discussion of that section.}.
	\end{itemize}
	
	Achieving these objectives presents several fundamental difficulties. For instance, establishing optimality requires both proving a general sharp upper bound on the convergence rate for a class of observables with homogeneous regularity and constructing at least one observable that attains an almost matching lower bound. Since classical techniques prove inadequate for the former, we introduce more delicate estimates concerning small divisors, which are detailed in Section \ref{SECNOV}. For the latter, we develop an effective and universal method to construct critical counterexamples by exploiting the near-resonances of almost every rotation. Furthermore, establishing optimality with respect to the initial points necessitates additional technical efforts, particularly when extending the result to hold for arbitrary initial points\footnote{We are grateful to Professor Bassam Fayad for his valuable comments concerning this point.}. We emphasize that while the main results of this paper---namely Theorem \ref{OPTT1}, Corollary \ref{OPTCORO1}, and Theorem \ref{OPTT2}---are stated in terms of toral translations, they apply equally to general dynamical systems that are smoothly conjugated to them; see the discussion in Section \ref{SEC31}.
	
	The remainder of this paper is organized as follows. Section \ref{PRE} introduces the basic definitions and notation fundamental to our analysis. Section \ref{STA} presents our main results concerning weighted Birkhoff averages: Theorem \ref{OPTT1} and Corollary \ref{OPTCORO1} establish lower bounds on the uniform convergence rates, while Theorem \ref{OPTT2} demonstrates the optimality of these results under various regularity conditions (finite differentiability, $C^\infty$ regularity, logarithmic $C^\infty$ regularity, and Gevrey regularity) in both discrete and continuous settings. The optimality is manifested in twenty-one aspects, which are detailed in Section \ref{SECOPT}. Since these conclusions hold for almost all rotations and specific (or even all) initial points, they may be viewed as universal results. The principal contributions of our work, alongside a comparative discussion with existing literature, are thoroughly elaborated in Section \ref{SECNOV}. Finally, Section \ref{SECPRO} is devoted to the proofs of the main results.

	\section{Preliminaries}\label{PRE}

	We begin with some basic definitions and notation concerning irrational rotations, approximants, and regularity, as well as the weighted Birkhoff averages of interest.
	
	To simplify the exposition, we assume $g(x)>0$ and adopt the following asymptotic notation to describe the behavior of functions as $x$ becomes sufficiently large (or small):
	\begin{itemize}
		\item  $f(x)=\mathcal{O}(g(x))$ means that there exists a constant $c>0$ such that $|f(x)| \leqslant c g(x)$. This is frequently abbreviated as $|f(x)|  \lesssim  g(x)$ or $g(x)  \gtrsim  |f(x)|$.
		
		\item  $f(x)=\mathcal{O}^{\#}(g(x))$ means that there exists a constant $c\geqslant 1$  such that $c^{-1}g(x)\leqslant f(x) \leqslant c g(x)$. 
		
		\item $f (x) \sim g(x)$ means that $f (x)$ and $g(x)$ are asymptotically equivalent; in other words, we have $\lim_{x \to +\infty}  {f(x)}/{g(x)} = 1$ (or $ \lim_{x \to 0}  {f(x)}/{g(x)} = 1 $).
		
		\item $f (x)=o(g(x))$ means that for any given $\varepsilon>0$, the inequality $|f (x)| \leqslant \varepsilon g(x)$ holds for all sufficiently large (or small) $x$. This is frequently abbreviated as $|f (x)| \ll g(x)$ or $g(x) \gg |f (x)|$.
	\end{itemize}
	
	Unless otherwise stated, constants may depend on the fixed weight, rotation vector, observable, dimension, and regularity parameters, but not on $N$, $T$, or the initial point $\theta$. All stated upper bounds are uniform in $\theta$. We use $\mathbb{N}=\{0,1,2,\ldots\}$ and $\mathbb{N}^+=\{1,2,\ldots\}$. We denote by $D^nf$ the $n$-th derivative of a sufficiently differentiable function $f$. We endow $\mathbb{R}^l$ ($l \in \mathbb{N}^+$) with the supremum norm $\|f\|_{\ell^\infty} := \max_{1 \leqslant j \leqslant l} |f_j|$\footnote{We emphasize that this $\ell^\infty$-norm is taken with respect to the vector components of the function rather than its spatial variables. Consequently, in the subsequent discussion, we will explicitly specify the underlying domain whenever this norm is invoked.}. Throughout this paper, $l$ denotes the dimension of the observables, while $d$ represents the dimension of the rotation vectors and initial points. Consequently, both $l$ and $d$ are positive integers.

	\subsection{Irrational rotations,  approximants, and regularity}\label{SUBSEC21}
	For $x \in \mathbb{R}$, let $\{x\}$ denote its fractional part, and $\|x\|_{\mathbb{Z}} := \operatorname{dist}(x, \mathbb{Z})$ denote its distance to the nearest integer. Then it is straightforward to verify that there exists an absolute constant $ \widetilde{C} \geqslant 1 $ such that
	\begin{equation}\label{zhishu}
		{{\widetilde C}^{ - 1}}{\left\| x \right\|_\mathbb{Z}} \leqslant \left| {{\exp\left({2\pi ix}\right)} - 1} \right| \leqslant \widetilde C{\left\| x \right\|_\mathbb{Z}}\leqslant\widetilde C \left|x\right|,\quad \forall x \in \mathbb{R}.
	\end{equation}
	Furthermore, this implies
	\begin{equation}\label{zhishu2}
		{\left\| {nx} \right\|_\mathbb{Z}} \leqslant \widetilde C  n {\left\| x \right\|_\mathbb{Z}},\quad \forall n\in \mathbb{N}, \quad \forall x \in \mathbb{R}.
	\end{equation}

	For $ d\geqslant 1 $, let $\mathbb{T}^d:=\mathbb{R}^d/\mathbb{Z}^d$ denote the standard $d$-torus. Consider an irrational number $ \rho  \in {\mathbb{T}}\setminus \mathbb{Q} $ with its continued fraction expansion 
	\[\rho  = \frac{1}{{{k_1} + \frac{1}{{{k_2} + \frac{1}{{\frac{ \cdots }{{{k_n} + \frac{1}{ \cdots }}}}}}}}}: = \left[ {{k_1},{k_2}, \ldots ,{k_n}, \ldots } \right],\]
	where the integers $k_n$ for $n \in \mathbb{N}^+$ are called the partial quotients of $ \rho $. Letting $ p_0=0 $ and $ q_0=1 $,  the $n$-th approximant of $ \rho $ is given by
	\[\frac{{{p_n}}}{{{q_n}}} = \left[ {{k_1},{k_2}, \ldots ,{k_n}} \right], \quad  n \geqslant 1.\]
	Here, $ p_n $ and $ q_n$ satisfy the following recurrence relations $ {p_{n + 1}} = {k_{n + 1}}{p_n} + {p_{n - 1}} $ and $ {q_{n + 1}} = {k_{n + 1}}{q_n} + {q_{n - 1}} $ for all $ n \geqslant 1 $. Observe that $ \{q_n\}_{n \in \mathbb{N}^+} $ is strictly increasing, which implies $ {q_{n + 2}} \geqslant 2{q_n} $. Consequently, we obtain the bounds $ {q_n} \geqslant {2^{\left( {n - 1} \right)/2}} $ and $ {q_{n + 2k}} \geqslant {2^k}{q_n} $ for all $ n,k \in \mathbb{N}^+ $. These inequalities are known as the \textit{exponential properties} of $ \{q_n\}_{n \in \mathbb{N}^+} $. Moreover, $ {\left\| {{q_n}\rho } \right\|_\mathbb{Z}} = \left| {{q_n}\rho  - {p_n}} \right| $, and 
	\begin{equation}\label{liangce}
		\frac{1}{{2{q_{n + 1}}}} < \left| {{q_n}\rho  - {p_n}} \right| < \frac{1}{{{q_{n + 1}}}},
	\end{equation}
	and
	\begin{equation}\label{liangce22}
		\|q_n\rho\|_\mathbb{Z}\leqslant\|q\rho\|_\mathbb{Z}, \quad  \forall 0<q<q_{n+1}.
	\end{equation}
	In this context, we recall a fundamental result in both ergodic theory and number theory: the \textit{Denjoy--Koksma inequality} (cf. \cite[Chapter VI]{Her79}).
	\begin{lemma}[Denjoy--Koksma inequality]\label{DENJOYKOKSMA}
		Let $ \rho  \in {\mathbb{T}}\setminus \mathbb{Q}  $ with its approximants	$ \{p_n/q_n\}_{n \in \mathbb{N}^+} $ be given.
		If $\mathcal{F}\colon \mathbb{T} \rightarrow \mathbb{R}$ is a function of bounded variation, with total variation denoted by $\operatorname{Var}(\mathcal{F})$, then
		\[\left|\sum_{j=0}^{q_n-1} \mathcal{F}\left(\theta +j \rho\right)-q_n \int_{\mathbb{T}} \mathcal{F}\left(x\right)  {d} x\right| \leqslant \operatorname{Var}(\mathcal{F}), \quad \forall n \in \mathbb{N}^+, \quad \forall  \theta \in \mathbb{T}.\]
	\end{lemma}

	As in the $ 1 $-dimensional case, one can also consider an irrational (nonresonant/rationally independent)   rotation vector $\rho=(\rho_1,\ldots,\rho_d)\in\mathbb{R}^d$ with $ d \geqslant 2$, that is, $\langle k,\rho\rangle:=\sum_{j=1}^d k_j\rho_j\notin\mathbb{Z}$ in the discrete case, and $\langle k,\rho\rangle\neq 0$ in the continuous case, for all $k=(k_1,\ldots,k_d)\in\mathbb{Z}^d\setminus\{0\}$.  In discrete time, $\rho$ may be taken modulo $\mathbb{Z}^d$; in continuous time it is a velocity in $\mathbb{R}^d$ and must not be reduced modulo $\mathbb{Z}^d$. For the continuous flow with $d=1$, nonresonance is equivalent to $\rho\ne0$, including nonzero rational velocities. These nonresonance conditions can be quantified through an \textit{approximation function}.
	
	\begin{definition}[Approximation function]
		A function $ \Delta $  is called an approximation function if it is continuous, strictly monotonically increasing, and satisfies $  \Delta(1) >0$ and $  \Delta(+\infty) =+\infty$.
	\end{definition}
	
	\begin{definition}[$ \Delta $-nonresonance condition]
		Given an approximation function $\Delta$, a vector $\rho \in \mathbb{R}^d$ with $d \geqslant 1$ is said to be $\Delta$-nonresonant if there exists a constant $\gamma > 0$ such that
		\begin{enumerate}[label=(\arabic*)] 
			\item   
			The discrete case: \begin{equation}\label{fgz}
				\| {\left\langle {k,\rho } \right\rangle  } \|_{\mathbb{Z}} \geqslant \frac{\gamma }{{\Delta \left( {\|k\|_{\ell^\infty}} \right)}}, \quad \forall   k \in {\mathbb{Z}^d}\setminus\{0\};
			\end{equation}
			\item  
			The continuous case: \begin{equation}\label{fgz2}
				| {\left\langle {k,\rho } \right\rangle  } | \geqslant \frac{\gamma }{{\Delta \left( {\|k\|_{\ell^\infty}} \right)}}, \quad \forall   k \in {\mathbb{Z}^d}\setminus\{0\}.
			\end{equation}
		\end{enumerate}
	\end{definition}
	\begin{remark}\label{fenli}
		In the discrete case, \eqref{fgz} holds for almost every\footnote{In the measure-theoretic sense, i.e., having full Lebesgue measure; such vectors are often said to be \textit{universal}.} $ \rho \in \mathbb{R}^d $ with $ \Delta (x)\sim x^{d}(\log x)^2 $.  In the continuous case, \eqref{fgz2} is satisfied for almost every $ \rho \in \mathbb{R}^d $ with $ \Delta (x)\sim x^{d-1}(\log x)^2 $.
	\end{remark}

	For a vector $\rho\in\mathbb R^d$ with $d\geqslant1$ satisfying the corresponding discrete or continuous nonresonance condition, we define the rotation $\mathscr T_\rho^s$ by \eqref{xuanzhuan}. In ergodic theory, it is natural to consider the Birkhoff averages of an observable $f$ on $\mathbb{T}^d$, given by
	\begin{equation}\label{BIK}
		\frac{1}{N}\sum\limits_{n = 0}^{N - 1} {f\left( {\mathscr{T}_\rho ^n\left( \theta  \right)} \right)} \quad\text{and} \quad  \frac{1}{T}\int_0^T {f\left( {\mathscr{T}_\rho ^s\left( \theta  \right)} \right)ds}, \quad \theta  \in {\mathbb{T}^d},
	\end{equation}
	with the aim of investigating their convergence rates to the spatial average $\int_{\mathbb{T}^d} f(x) dx$. Alongside the deterministic convergence results for Birkhoff averages discussed in the Introduction (which are closely related to Denjoy--Koksma-type inequalities), we also note the significant progress made in the random setting \cite{DF14,DF15,DF20}. The present paper, however, focuses exclusively on the deterministic case, leaving the random setting for future investigation.
	
	To characterize  the explicit regularity for observables, we introduce the class $\mathbb{R}_{\widetilde{\Delta}}^{l}(\mathbb{T}^d)$ as follows.

	\begin{definition}[Fourier decay class $\mathbb{R}_{\widetilde{\Delta}}^l(\mathbb{T}^d)$]\label{DEFBANACH}
		Let $d,l$ be positive integers and let $\widetilde\Delta$ be an approximation function. The class $\mathbb{R}_{\widetilde\Delta}^l(\mathbb{T}^d)$ consists of continuous observables $f\colon\mathbb{T}^d\to\mathbb{R}^l$ whose Fourier coefficients
		\[
		\widehat f(k)=\int_{\mathbb{T}^d}f(x)\exp(-2\pi i\langle k,x\rangle)dx
		\]
		satisfy
		\[
		\sum_{k\in\mathbb{Z}^d\setminus\{0\}}\|\widehat f(k)\|_{\ell^\infty}<+\infty
		\]
		and
		\[
		\|\widehat f(k)\|_{\ell^\infty}
		\leqslant\widetilde\Delta(\|k\|_{\ell^\infty})^{-1},
		\quad\forall k\in\mathbb{Z}^d\setminus\{0\}.
		\]
	\end{definition}
	The Fourier series converges absolutely and uniformly, and
	\[
	\|f(x)-\widehat f(0)\|_{\ell^\infty}
	\leqslant\sum_{k\ne0}\|\widehat f(k)\|_{\ell^\infty}<+\infty.
	\]
	Only the actual Fourier coefficients are required to be absolutely summable. Summability of the full envelope $\sum_{k\ne0}\widetilde\Delta(\|k\|_{\ell^\infty})^{-1}$ is a sufficient condition for this property, but is not imposed in the definition. In particular, the sparse observables in Theorem \ref{OPTT1} are admissible whenever $\widetilde\Delta(x)\geqslant x^\varepsilon$ for all sufficiently large $x$ with any $\varepsilon>0$, including $0<\varepsilon\leqslant d$: the selected denominators grow at least exponentially, so $\sum_j\widetilde\Delta(q_{m_j})^{-1}<+\infty$.
	The coefficient bound is normalized for convenience. Observables satisfying the same bound with an additional fixed factor $C>0$ are treated by scaling, which multiplies the upper error bounds by $C$. The condition $\widetilde\Delta(x)\sim x^L$ prescribes polynomial Fourier decay, rather than an equivalent definition of $C^L$ regularity. It implies $C^r$ regularity for integers $r\geqslant0$ with $L>r+d$. If $\widetilde\Delta(x)\gg x^L$ for every $L>0$, every observable in the class is $C^\infty$. Here and in what follows, all asymptotic properties concerning $\Delta$ and $\widetilde\Delta$ are understood to be as $x \to +\infty$. Beyond $C^\infty$ regularity, there exists a class of regularity known as ultra-differentiable regularity, which has attracted considerable interest in the theory of smooth dynamical systems. Specifically, one can consider the following examples: $\widetilde{\Delta}(x) \sim \exp((\log x)(\log \log x)^\sigma)$ with $\sigma > 0$, $\widetilde{\Delta}(x) \sim \exp((\log x)^{\lambda})$ with $\lambda > 1$, and $\widetilde{\Delta}(x) \sim \exp(x^\alpha)$ with $\alpha > 0$. It is worth noting that the first two growth rates are asymptotically very close to polynomial growth. In particular, the second case can be referred to as \textit{logarithmic $C^\infty$ regularity} \cite{TL26b}, whereas the last case corresponds to the well-known \textit{$\alpha$-Gevrey regularity} for $0 < \alpha < 1$ \cite{Pop04,BF19,BF21,TL26b}, reducing to analyticity when $\alpha = 1$. Such cases are classical examples of Hardy field functions \cite{Har71}. In this context, we refer to \cite{Fra15,FH18b,BMR20,Fra22,Ric23,Tsi23} and the references therein for their applications in ergodic theory. We remark at the outset that although one of the main results of the present paper, namely Theorem \ref{OPTT2}, restricts its consideration of ultra-differentiability solely to the logarithmic $C^\infty$ and $\alpha$-Gevrey classes, the underlying techniques can be adapted to other specific ultra-differentiable classes, yielding optimal results. Such extensions are left to the interested reader. Furthermore, from an ergodic-theoretic perspective, the results of the present paper may provide novel insights into the aforementioned literature within the setting of ultra-differentiable regularity.

	\subsection{The weighted Birkhoff average}\label{SUBSECWBA}

	To investigate \textit{general} compactly supported weighting functions, we first extract and generalize the inherent quantitative properties of Laskar's function ${w}_{\rm Las}\left( x \right)$ in \eqref{weightingfunction} and its parameterized variant ${w}_{\rm Las}^{p,q}\left( x \right)$ in \eqref{bianti}. This generalization is \textit{essential} for establishing optimal convergence rates for weighted Birkhoff averages.

	\begin{definition}[Weighting function]\label{generalwei}
		Let $w $ be a non-negative continuous weighting function defined on $\mathbb{R}$, satisfying $w(x)=0$ for $x \in \mathbb{R}\setminus [0,1]$ and the normalization condition $\int_0^1 w\left(x\right) dx = 1$.
		\begin{enumerate}[label=(\arabic*)] 
			\item \label{item:num_1} We say that $w  \in C_0^M([0,1])$ with $M \in \mathbb{N}^+ \cup \{+\infty\}$, if $w\in C^M(\mathbb{R})$ and
			\begin{equation*}  
				D^j w(0) = D^j w(1) = 0, \quad \forall 0 \leqslant j \leqslant M. 
			\end{equation*}
			\item \label{item:num_2} Furthermore, we say that $w  \in \mathcal{W}_\beta$ for $\beta > 1$, if $w$ satisfies Condition \ref{item:num_1} with $M=+\infty$, and there exist constants $\bar{C}, \lambda > 0$ such that
			\begin{equation*}  
				\| D^m w  \|_{L^1(0,1)} \leqslant \bar{C} \lambda^m m^{m\beta}, \quad \forall m \in \mathbb{N}^+.
			\end{equation*}
		\end{enumerate}
	\end{definition}

	For Laskar's function ${w}_{\rm Las}\left( x \right)$, the asymptotic property described in Condition \ref{item:num_2} was originally formulated qualitatively in \cite[Lemma 5.3]{TL24a} for some $\beta>6$, while a more quantitative version for its parameterized variant ${w}_{\rm Las}^{p,q}\left( x \right)$ was provided in \cite[Lemma 4.1]{TL25b} with $\beta=1+(\min\{p,q\})^{-1}$. It should be noted that Condition \ref{item:num_2} is \textit{not} an artificial construction; rather, it plays a significant role in various other problems within dynamical systems. For instance, a stronger pointwise version of this condition with $\beta=2$ can be found in \cite{Eli97}\footnote{Verification via Stirling's formula is left to the reader.}. In this paper, we focus on the optimal ergodic rates \textit{primarily} associated with Condition \ref{item:num_2}; although Condition \ref{item:num_1} is of secondary interest, certain results remain optimal in this context as well (see Case \ref{item3:roman_I}--\ref{TH3.3-I-1} in Theorem \ref{OPTT2}). In particular, the results of this paper correspond to the case $\beta=2$ when applied directly to Laskar's function ${w}_{\rm Las}\left( x \right)$.
	
	Throughout this paper, we apply the weighting function $w(x)$ from Definition \ref{generalwei} to the classical Birkhoff averages on the torus \eqref{BIK}, yielding the  (non-uniformly) \textit{weighted Birkhoff averages}. When introducing weighting to the averages, it is essential to ensure that the original limits are preserved. Otherwise, the weighted scheme would lack justification, especially in the context of accelerating convergence rates. Theorem 1.1 in \cite{TL26a}, which serves as an abstract averaging version of the Toeplitz theorem, guarantees this preservation. Consequently, by the unique ergodicity of the irrational rotation, we immediately deduce the following limit-preserving proposition.

	\begin{proposition}\label{fppro}
		Let $d$, $l$, and $M$ be positive integers.
		Suppose $w \in C_0^M([0,1])$,
		$\rho \in \mathbb{R}^d$ is a nonresonant vector
		in the corresponding discrete or continuous
		sense defined above,
		and $f \colon \mathbb{T}^d \to \mathbb{R}^l$ is a continuous
		observable.
		Then, for any initial point $\theta \in \mathbb{T}^d$,
		the following identities hold in the $\ell^\infty$ sense:
		
		In the discrete case,
		\[	\lim_{N \to +\infty}
		\frac{1}{A_N}
		\sum_{n=0}^{N-1}
		w\left(n/N\right)
		f\left(\mathscr{T}_\rho^n(\theta)\right)
		=
		\lim_{N \to +\infty}
		\frac{1}{N}
		\sum_{n=0}^{N-1}
		f\left(\mathscr{T}_\rho^n(\theta)\right)= \int_{\mathbb{T}^d} f(x)dx .\]
		In the continuous case,
		\[	\lim_{T \to +\infty}
		\frac{1}{T}
		\int_0^T
		w\left(s/T\right)
		f\left(\mathscr{T}_\rho^s(\theta)\right) ds
		=
		\lim_{T \to +\infty}
		\frac{1}{T}
		\int_0^T
		f\left(\mathscr{T}_\rho^s(\theta)\right) ds  
		= \int_{\mathbb{T}^d} f(x) dx .\]
	\end{proposition}

	Since we work with the weighting function $w(x)$ given in Definition \ref{generalwei}, in view of Proposition \ref{fppro} and to simplify the exposition, let us define the errors of the weighted Birkhoff averages in the discrete and  continuous cases, respectively (for sufficiently large $ N $ and $ T $):
	\[\mathbf{Error}_N^{\mathrm{dis}}\left(w,f,\rho,\theta\right): = {\left\| {\frac{1}{{{A_N}}}\sum\limits_{n = 0}^{N - 1} {w\left( {n/N} \right)f\left( {\mathscr{T}_\rho ^n\left( \theta  \right)} \right)}  - \int_{\mathbb{T}^d} {f\left( x \right)dx} } \right\|_{{\ell ^\infty }}}, \quad \theta \in \mathbb{T}^d,\]
	and 
	\[\mathbf{Error}_T^{\mathrm{con}}\left(w,f,\rho,\theta\right) := {\left\| {\frac{1}{T}\int_0^T {w\left( {s/T} \right)f\left( {\mathscr{T}_\rho ^s\left( \theta  \right)} \right)ds}  - \int_{{\mathbb{T}^d}} {f\left( x \right)dx} } \right\|_{{\ell ^\infty }}},\quad \theta \in \mathbb{T}^d.\]
	It should be emphasized that these errors depend on the initial point $\theta$; this explicit dependence will be demonstrated in Section \ref{STA}.

	For irrational rotations and sufficiently smooth observables, it is possible to achieve \textit{rapid convergence} for weighted Birkhoff averages. Specifically, we refer to \cite[Theorem 2.6]{TL24a} for a typical result concerning the class $\mathbb{R}^l_{\widetilde{\Delta}}(\mathbb{T}^d)$, which is formulated as follows.

	\begin{theorem}\label{OPTYINDL}
		Let $d$, $l$, and $M$ be positive integers. Let $w \in C_0^M([0,1])$ and $f \in \mathbb{R}^l_{\widetilde{\Delta}}(\mathbb{T}^d)$. Suppose $\rho \in \mathbb{R}^d$ is $\Delta$-nonresonant in the corresponding discrete or continuous sense, where $\Delta$ and $\widetilde{\Delta}$ are given approximation functions. Assume also  that there exists an integer $2\leqslant m\leqslant M$ such that the following integrability condition holds:
		\begin{equation}\label{keji}
			\int_1^{ + \infty } {\frac{{{r^{d - 1}}{\Delta ^m}\left( r \right)}}{{\widetilde \Delta \left( r \right)}} {d}r}  <  + \infty .
		\end{equation}
		Then the discrete and continuous conditions imply, respectively, the corresponding estimates below as $N\to+\infty$ and $T\to+\infty$:
		\begin{equation*} 
			\mathbf{Error}_N^{\mathrm{dis}}\left(w,f,\rho,\theta\right) =\mathcal{O} (N^{-m} )\quad \text{and}\quad  \mathbf{Error}_T^{\mathrm{con}}\left(w,f,\rho,\theta\right)  =\mathcal{O} (T^{-m} ),\quad \forall \theta  \in \mathbb{T}^d .
		\end{equation*}
	\end{theorem}

	Indeed, by employing approaches similar to those in \cite{TL25b,TL26a}, the uniform convergence with respect to the initial point $\theta$ could be established in a stronger sense than that of Theorem \ref{OPTYINDL}. Beyond convergence in the supremum norm, the weighted Birkhoff average might also converge in the $C^k$ norm, the Gevrey norm, or even the analytic norm. For instance, consider a Diophantine rotation vector $\rho\in\mathbb{R}^d$ with exponent $\tau>d\geqslant1$ (corresponding to the $x^{\tau}$-nonresonance condition in \eqref{fgz}) and an observable extending analytically to the complex strip $\left\|\operatorname{Im}x\right\|_{\ell^\infty}\leqslant\mathscr{R}$, endowed with the analytic norm defined by
	\[\|f\|_{\mathscr{A},\mathscr{R}}:=\sum_{k\in\mathbb{Z}^d}\|\widehat{f}(k)\|_{\ell^\infty}\exp(2\pi\mathscr{R}\|k\|_{\ell^1})<+\infty,\quad\mathscr{R}>0.\]
	Here $\|k\|_{\ell^1}=\sum_{j=1}^d|k_j|$. Under these conditions, and assuming $w\in\mathcal W_\beta$, the weighted Birkhoff average converges at an exponential rate in the aforementioned norm on the narrower strip $\left\|\operatorname{Im}\theta\right\|_{\ell^\infty}\leqslant\mathscr{R}-\epsilon$ for any $0<\epsilon <\mathscr{R}$. This narrowing of the domain reflects the loss of regularity caused by small divisors, a phenomenon well known in KAM theory; see, for instance, \cite{Pop04,Sal04,Bou05,BF19,BMP20,BF21,MP21,TL26b} and the references therein.

	We conclude this section by noting that, beyond convergence acceleration, various weighted Birkhoff averages have been explored for other purposes, exhibiting slower convergence rates; see, for instance, \cite{KT03,GT12,Bou13,BSZ13,LS15,HLSY17,Wan17,FJ18,FH18a,Fan19,HWY19,HWZ19,Bog20,Fan21,AW23,PR26} and the references therein.
	\section{Main results: Optimal uniform convergence rates for weighted Birkhoff averages}\label{STA}

	\subsection{Main results}\label{SEC31}
	In this section, we present three principal results concerning the weighted Birkhoff averages introduced in Section \ref{SUBSECWBA}. Theorem \ref{OPTT1} (Main Theorem I) establishes lower bounds for the convergence rates of these averages, for the discrete case on $\mathbb{T}$ and the continuous case on $\mathbb{T}^2$. Corollary \ref{OPTCORO1} then extends this result to $\mathbb{T}^d$, subject to certain necessary modifications. Finally, Theorem \ref{OPTT2} (Main Theorem II) addresses the optimality of these convergence rates by simultaneously investigating the corresponding upper bounds across four typical regularity settings: finite differentiability and its variants, $C^\infty$ regularity, logarithmic $C^\infty$ regularity, and Gevrey regularity.

	All the aforementioned results systematically address both discrete and continuous frameworks, which exhibit fundamental differences. Theorem \ref{OPTT1}, Corollary \ref{OPTCORO1}, and the vast majority of the assertions within Theorem \ref{OPTT2} constitute entirely novel contributions, in terms of both their conclusions and the underlying technical methodologies; only a small portion of Theorem \ref{OPTT2} consists of established historical results or variations thereof. It should be emphasized that these results establish, for the first time, lower bound estimates for the convergence rates of (non-uniformly) weighted Birkhoff averages, and furthermore demonstrate their \textit{optimality}.
	
	The present section is devoted exclusively to the precise statements of these results alongside their interpretations. The optimality of these results is illustrated in Section \ref{SECOPT}, while a comprehensive exposition of our primary contributions and a comparison with previous literature are deferred to Section \ref{SECNOV}.

	\begin{theorem}[\textbf{Main Theorem I}]\label{OPTT1}
		Let $w  \in \mathcal{W}_{\beta}$ be a weighting function with $\beta>1$, and let $\widetilde{\Delta} $ be an approximation function satisfying $x^\varepsilon \leqslant \widetilde{\Delta}(x) \leqslant \exp(x^\zeta)$ for some $\varepsilon > 0$ and $0 < \zeta < \beta^{-1}$, for all sufficiently large $x$. Then the following statements hold:
		\begin{enumerate}[label=(\Roman*)] 
			\item \label{item:roman_I} For almost every $\rho \in \mathbb{R}$, there exists an observable $\Psi \in \mathbb{R}_{\widetilde{\Delta}}(\mathbb{T} )$ and a dense set $\Xi \subset \mathbb{T}$ with $0 \in \Xi$ such that for every $\theta \in \Xi$, the following inequality holds for infinitely many $N\in\mathbb{N}^+$:
			\[	\mathbf{Error}_N^{\mathrm{dis}}\left(w,\Psi,\rho,\theta\right) \geqslant \frac{1}{{4\widetilde \Delta \left( N \right)}}.\]
			Similarly, for almost every $\rho \in \mathbb{R}^2$,	there exists an observable $\Phi \in \mathbb{R}_{\widetilde{\Delta}}(\mathbb{T}^2 )$ and a dense set $\Theta \subset \mathbb{T}^2$ with $(0,0) \in \Theta$ such that for every $\theta \in \Theta$, the following inequality holds for a sequence   $T \to +\infty$:
			\[\mathbf{Error}_T^{\mathrm{con}}\left(w,\Phi,\rho,\theta\right) \geqslant \frac{1}{{4\widetilde \Delta \left( T \right)}}.\] 
			\item  \label{item:roman_II} For almost every $\rho \in \mathbb{R}$, there exists an observable $\Psi^* \in \mathbb{R}_{\widetilde{\Delta}}^2 (\mathbb{T} )$  such that  the following inequality holds for infinitely many $N\in\mathbb{N}^+$:
			\[	\mathbf{Error}_N^{\mathrm{dis}}\left(w,\Psi^*,\rho,\theta\right) \geqslant \frac{1}{{4\widetilde \Delta \left( N \right)}} ,\quad \forall   \theta \in \mathbb{T}. \] 		
			Similarly, 	for almost every $\rho \in \mathbb{R}^2$, there exists an observable 	$\Phi^* \in \mathbb{R}_{\widetilde{\Delta}}^2 (\mathbb{T}^2 )$  such that the following inequality holds for a sequence $T \to +\infty$:
			\[ 	\mathbf{Error}_T^{\mathrm{con}}\left(w,\Phi^*,\rho,\theta\right) \geqslant \frac{1}{{4\widetilde \Delta \left( T \right)}},\quad \forall   \theta \in \mathbb{T}^2.\]	  
			\item  \label{item:roman_III} For almost every $\rho \in \mathbb{R}$ and every $ \theta \in \mathbb{T} $, there exists an observable $\Psi_* \in \mathbb{R}_{\widetilde{\Delta}} (\mathbb{T} )$  such that  the following inequality holds for infinitely many $N\in\mathbb{N}^+$:
			\[	\mathbf{Error}_N^{\mathrm{dis}}\left(w,\Psi_*,\rho,\theta\right) \geqslant \frac{1}{{4\widetilde \Delta \left( N \right)}}.\] 
			Similarly, for almost every $\rho \in \mathbb{R}^2$ and every $ \theta \in \mathbb{T}^2 $,	there exists an observable  $\Phi_* \in \mathbb{R}_{\widetilde{\Delta}} (\mathbb{T}^2 )$ such that  the following inequality holds for a sequence $T \to +\infty$:
			\[\mathbf{Error}_T^{\mathrm{con}}\left(w,\Phi_*,\rho,\theta\right) \geqslant \frac{1}{{4\widetilde \Delta \left( T \right)}}.\]
		\end{enumerate}
		Indeed, the observables in Cases \ref{item:roman_I} and \ref{item:roman_III} can be chosen from the class ${\mathbb{R}}_{\widetilde\Delta}^l(\cdot)$ for any $l\geqslant 1$, whereas those in Case \ref{item:roman_II} can be chosen from ${\mathbb{R}}_{\widetilde\Delta}^l(\cdot)$ for any $l\geqslant 2$.
	\end{theorem}
	\begin{remark}\label{RE31}
		For the sake of clarity, we present the explicit constructions prior to the proof. For almost every   $\rho \in \mathbb{R}$, the observables $\Psi$ and $\Psi^*$ in Cases \ref{item:roman_I} and \ref{item:roman_II} can be constructed as
		\[\Psi \left( \theta  \right) = \operatorname{Re} \sum\limits_{j = 1}^\infty  {\frac{{{\exp\left({2\pi i{q_{{m_j}}}\theta }\right)}}}{{\widetilde \Delta \left( {{q_{{m_j}}}} \right)}}}\quad \text{and} \quad {\Psi ^*}\left( \theta  \right) = {\left( {\operatorname{Re} \sum\limits_{j = 1}^\infty  {\frac{{\exp \left( {2\pi i{q_{{m_j}}}\theta } \right)}}{{\widetilde\Delta \left( {{q_{{m_j}}}} \right)}}} ,\;\operatorname{Im} \sum\limits_{j = 1}^\infty  {\frac{{\exp \left( {2\pi i{q_{{m_j}}}\theta } \right)}}{{\widetilde\Delta \left( {{q_{{m_j}}}} \right)}}} } \right)^ \top } \]
		for a suitably chosen increasing sequence of positive integers $ \{m_j\}_{j\in \mathbb{N}^+} $, where $ q_{m_j} $ is the  denominator of the $ {m_j} $-th  approximant of $ \rho $. Meanwhile, the observable $\Psi_* $ in Case \ref{item:roman_III} can be constructed as one of the two components of $\Psi^*$, depending on the value of $\theta$.  Moreover, the set $\Xi$ in Case \ref{item:roman_I} can be constructed as the set of fractional parts $\{ \{\iota\rho\} \}_{\iota \in \mathbb{N} }$, whose density in $ \mathbb{T} $ is a direct consequence of  Kronecker's theorem.
	\end{remark}
	\begin{remark}\label{rem:upper-growth}
		The upper-growth restriction on $\widetilde\Delta$ in Theorem \ref{OPTT1} is retained in the lower-bound conclusions derived from it in this paper. Its removal by alternative techniques, with an emphasis on existence rather than the explicit constructions given here, will be discussed in a forthcoming work.
	\end{remark}

	By virtue of Case \ref{item:roman_I} in Theorem \ref{OPTT1}, we directly obtain the following Corollary \ref{OPTCORO1} for the higher-dimensional torus $\mathbb{T}^d$. The lower bounds hold on dense cylinders of initial points, since the embedded observables do not depend on the remaining coordinates. Analogously, other versions can be deduced from Cases \ref{item:roman_II} and \ref{item:roman_III} in Theorem \ref{OPTT1}; these are omitted here for brevity.
	
	\begin{corollary}\label{OPTCORO1}
		Let $w  \in \mathcal{W}_{\beta}$ be a weighting function with $\beta>1$,  let $\widetilde{\Delta} $ be an approximation function satisfying $x^\varepsilon \leqslant \widetilde{\Delta}(x) \leqslant \exp(x^\zeta)$ for some $\varepsilon > 0$ and $0 < \zeta < \beta^{-1}$, for all sufficiently large $x$, and let $ l \geqslant 1 $. Then the following statements hold: First, for almost every $\rho \in \mathbb{R}^d$ with $  d \geqslant 2 $, there exists an  observable $ \bar \Psi \in \mathbb{R}_{\widetilde{\Delta}}^l(\mathbb{T}^d) $   and a dense set $ \bar \Xi  \subset {\mathbb{T} }  $ with $0 \in \bar \Xi$ such that  for every $\theta=(\widetilde\theta,\theta')\in\bar\Xi\times\mathbb T^{d-1}$, the following inequality holds for infinitely many $N\in\mathbb{N}^+$:
		\[\mathbf{Error}_N^{\mathrm{dis}}\left(w,\bar \Psi,\rho,\theta\right) \geqslant \frac{1}{{4\widetilde \Delta \left( N \right)}} .\]
		Similarly, 	for almost every $\rho \in \mathbb{R}^d$ with $  d \geqslant 3 $, there exists an  observable $ \bar \Phi \in \mathbb{R}_{\widetilde{\Delta}}^l(\mathbb{T}^d) $   and a dense set $ \bar\Theta  \subset {\mathbb{T}^2 }  $ with $(0,0) \in\bar \Theta$ such that  for every $\theta=(\widetilde\theta_1,\widetilde\theta_2,\theta')\in\bar\Theta\times\mathbb T^{d-2}$, the following inequality holds for a sequence   $T \to +\infty$:
		\[\mathbf{Error}_T^{\mathrm{con}}\left(w,\bar \Phi,\rho,\theta\right) \geqslant \frac{1}{{4\widetilde \Delta \left( T \right)}}.\] 
		The initial-point sets $\bar\Xi\times\mathbb T^{d-1}$ and
		$\bar\Theta\times\mathbb T^{d-2}$ are dense in $\mathbb T^d$
		and contain the origin. The lower-bound times may depend on the
		coordinate or coordinate pair in the base set, but are independent
		of the remaining coordinates.
	\end{corollary}
	\begin{remark}
		It is worth noting that Corollary \ref{OPTCORO1} could potentially be improved by leveraging the theory of multidimensional continued fractions (which are not uniquely defined) in conjunction with Theorem \ref{OPTT1}; this will be the subject of future work.
	\end{remark}
	\begin{remark}
		An analogous conclusion applies to the infinite-dimensional  torus $\mathbb{T}^\mathbb{Z}$, following the framework in \cite{Bou05,BMP20,MP21,TL24a,TL26a}, for instance. However, for the sake of conciseness and clarity, we restrict our attention to the finite-dimensional case in this paper.
	\end{remark}

	Finally, we are in a position to present Theorem \ref{OPTT2}, establishing the \textit{optimality} of the uniform convergence rates for weighted Birkhoff averages for almost all  rotations and for observables with prescribed regularity, encompassing four typical classes. Theorem \ref{OPTT2} incorporates both negative and positive results. The former are derived from Theorem \ref{OPTT1} and Corollary \ref{OPTCORO1}, whereas the latter rely on novel analytical techniques---detailed in Section \ref{SECNOV}---to attain the optimal rates. As will become evident, \textit{Theorem \ref{OPTT2} demonstrates that the regularity of the observables plays a fundamentally crucial role in determining the convergence rates for weighted Birkhoff averages in the general setting.} The twenty-one aspects of optimality it exhibits will be comprehensively discussed in Section \ref{SECOPT}.

	\begin{theorem}[\textbf{Main Theorem II}]\label{OPTT2}
		Let $w \in \mathcal{W}_{\beta}$ be a weighting function with $\beta>1$, let $\widetilde \Delta$ be an approximation function, and let $l \geqslant 1$. Then the following statements hold:
		\begin{enumerate}[label=(\Roman*)]
			\item \label{item3:roman_I} The finitely differentiable case on $\mathbb{T}^d$: 
			\begin{enumerate}[label=(\roman*)] 
				\item \label{TH3.3-I-1} Suppose that $\widetilde\Delta(x)\sim x^\ell$ with $\ell>d\geqslant1$.
				Then, for almost every $\rho\in\mathbb R^d$ and every
				$f\in\mathbb R^l_{\widetilde\Delta}(\mathbb T^d)$, as $N\to+\infty$, we   have
				\[
				\mathbf{Error}_N^{\mathrm{dis}}(w,f,\rho,\theta)
				=\mathcal O(N^{-\ell'}),\quad \forall\theta\in\mathbb T^d
				\]
				for any $1<\ell'<d^{-1}\ell$.
				However, in general\footnote{By ``general,'' we mean that for almost all $\rho\in\mathbb{R}^d$, Theorem \ref{OPTT1} and Corollary \ref{OPTCORO1} provide the desired counterexamples for specific (or even all) initial points; similar arguments apply throughout.}, the error cannot decay faster than $\mathcal{O}(N^{-\ell})$ for all cases $\ell > 0$ with $  d \geqslant 1$. Moreover, these results hold for any weighting function $w \in C_0^M([0,1])$ with $M \geqslant \max\{\ell, 2\}$.
				
				\item \label{TH3.3-I-2} Suppose that $\widetilde\Delta(x) \sim x^\ell$ with $\ell > d\geqslant 1$. Then, for almost every $\rho \in \mathbb{R}^d$  and every $f \in \mathbb{R}^l_{\widetilde{\Delta}}(\mathbb{T}^d)$, as $T \to +\infty$, we have\footnote{To establish a parallel with the discrete setting in Cases \ref{item3:roman_II}--\ref{item3:roman_IV}, we elect to provide estimates for the continuous case that coincide with their discrete counterparts for all $d \geqslant 1$. However, please keep in mind that a faster convergence rate is indeed recovered when $d=1$, as demonstrated herein.}
				\[
				\mathbf{Error}_T^{\mathrm{con}}(w, f, \rho, \theta) = 
				\begin{cases}
					o\bigl(\exp(-T^{\varrho})\bigr)\quad \text{for any}\quad 0<\varrho <\beta^{-1}, & d = 1, \\
					\mathcal{O}(T^{-\ell'})\quad \text{for any}\quad 1 < \ell' < (d-1)^{-1}\ell, & d \geqslant 2,
				\end{cases}\quad\forall \theta \in \mathbb{T}^d.
				\]
				Indeed, when $d=1$, the stated upper bound holds for every real velocity $\rho\ne0$, including nonzero rational velocities.
				However, in general, the error cannot decay faster than $\mathcal{O}(T^{-\ell})$ for all cases $ \ell>0 $ with $d \geqslant 2$.

				\item \label{TH3.3-I-3}  Suppose that $ \widetilde \Delta \left( x \right) \sim {x^\ell }{\left( {\log x} \right)^\eta } $  with $ \ell  \geqslant  2 $ and $ \eta  > 2\ell  + 1 $. Then, for almost every $ \rho \in \mathbb{R}^2 $ and every $ f \in   \mathbb{R}^l_{\widetilde{\Delta}}(\mathbb{T}^2) $, as $ T \to +\infty $, we have
				\[\mathbf{Error}_T^{\mathrm{con}}\left(w,f,\rho,\theta\right)   = \mathcal{O} (T^{-\ell} ), \quad \forall \theta \in \mathbb{T}^2.\]
				However, in general, the error cannot decay faster than $ \mathcal{O}({T^{ - \ell }}{\left( {\log T} \right)^{ - \eta }}) $ for all cases where $ \ell>0 $ and $ \eta \in \mathbb{R} $. 
			\end{enumerate}
			\item \label{item3:roman_II} The $C^\infty$ case on $\mathbb{T}^d$:\quad Suppose that $\widetilde\Delta(x) \gg x^L$ for any $L > 0$. Then, for almost every $\rho \in \mathbb{R}^d$ and every $f \in \mathbb{R}^l_{\widetilde{\Delta}}(\mathbb{T}^d)$ with $d\geqslant 1$, as $N \to +\infty$ and $T \to +\infty$, we have
			\[
			\mathbf{Error}_N^{\mathrm{dis}}(w, f, \rho, \theta) = \mathcal{O}(N^{-m}) \quad \text{and} \quad \mathbf{Error}_T^{\mathrm{con}}(w, f, \rho, \theta) = \mathcal{O}(T^{-m}), \quad \forall \theta \in \mathbb{T}^d
			\]
			for any $m > 0$. For the following lower bounds, assume additionally that $\widetilde\Delta(x)\leqslant\exp(x^\zeta)$ for all sufficiently large $x$, for some $0<\zeta<\beta^{-1}$, as in Theorem \ref{OPTT1}. Under this additional assumption, in general, the errors cannot decay faster than $\mathcal{O}(\widetilde{\Delta}(N)^{-1})$ for $d \geqslant 1$ and $\mathcal{O}(\widetilde{\Delta}(T)^{-1})$ for $d \geqslant 2$.
			
			\item \label{item3:roman_III}
			The logarithmic $C^\infty$ case on $\mathbb{T}^d$:\quad Suppose that $\widetilde{\Delta}(x) \sim \exp((\log x)^{\lambda})$ with $\lambda > 1$. Then, for almost every $\rho \in \mathbb{R}^d$ and every $f \in \mathbb{R}^l_{\widetilde{\Delta}}(\mathbb{T}^d)$ with $d\geqslant 1$, as $N \to +\infty$ and $T \to +\infty$, we have
			\[
			\mathbf{Error}_N^{\mathrm{dis}}(w, f, \rho, \theta) = \mathcal{O}\bigl(\exp(-c(\log N)^{\lambda})\bigr),\quad \forall \theta \in \mathbb{T}^d ,\]
			and
			\[ \mathbf{Error}_T^{\mathrm{con}}(w, f, \rho, \theta) = \mathcal{O}\bigl(\exp(-c(\log T)^{\lambda})\bigr), \quad \forall \theta \in \mathbb{T}^d,
			\]
			for some constant $0<c<1$. However, in general, the errors cannot decay faster than $\mathcal{O}\bigl(\exp(-(\log N)^\lambda)\bigr)$ for $d \geqslant 1$ and $\mathcal{O}\bigl(\exp(-(\log T)^\lambda)\bigr)$ for $d \geqslant 2$.

			\item \label{item3:roman_IV} The $\alpha$-Gevrey case on $\mathbb{T}^d$:\quad Suppose that $\widetilde\Delta(x) \sim \exp(x^\alpha)$ with $ \alpha >0$. Then, for almost every $\rho \in \mathbb{R}^d$ and every $f \in \mathbb{R}^l_{\widetilde{\Delta}}(\mathbb{T}^d)$ with $d\geqslant 1$, as $N \to +\infty$ and $T \to +\infty$, we have
			\[
			\mathbf{Error}_N^{\mathrm{dis}}(w, f, \rho, \theta) = o\bigl(\exp(-N^\xi)\bigr) \quad \text{and} \quad \mathbf{Error}_T^{\mathrm{con}}(w, f, \rho, \theta) = o\bigl(\exp(-T^\xi)\bigr), \quad \forall \theta \in \mathbb{T}^d
			\]
			for any $0 < \xi < \alpha(\alpha\beta + d)^{-1}$. However, in general, the errors cannot decay faster than $\mathcal{O}\bigl(\exp(-N^{\alpha})\bigr)$ for $d \geqslant 1$ and $\mathcal{O}\bigl(\exp(-T^{\alpha})\bigr)$ for $d \geqslant 2$, provided that $0 <\alpha < \beta^{-1}$.
		\end{enumerate}
	\end{theorem}

	\begin{remark}\label{RE37}
		Note that in Cases \ref{item3:roman_I}--\ref{TH3.3-I-1} to \ref{item3:roman_I}--\ref{TH3.3-I-3}, the assumptions regarding the lower bounds of the convergence rates are weaker than those for the upper bounds. This asymmetry arises because the specific observables constructed in Theorem \ref{OPTT1} are well-defined, whereas the general upper-bound argument uses summability of the full Fourier envelope; see Section \ref{PFTHI} for further details. Conversely, in Case \ref{item3:roman_IV}, the assumptions required for the upper bounds are weaker than those for the lower bounds. This is because higher regularity facilitates upper bound estimates, yet hinders the establishment of lower bounds. Consequently, certain restrictions must be imposed on the Gevrey index in Theorem \ref{OPTT1}, and hence in Case \ref{item3:roman_IV}. Likewise, in Case \ref{item3:roman_II}, the lower bounds require the stated upper-growth condition on $\widetilde\Delta$, whereas the upper bounds do not; see Remark \ref{rem:upper-growth}.
	\end{remark}

	\begin{remark}
		Note that Case \ref{item3:roman_I} is not equivalent to the classical $C^\ell(\mathbb{T}^d)$ setting, although classical $C^r$ regularity gives the Fourier bound of order $r$ for an integer $r$ (see \cite[p. 200]{Gra14}); when $r>d$, the coefficients are also absolutely summable. Therefore, we employ the straightforward asymptotic properties of Fourier coefficients to establish the convergence rates. As for Case \ref{item3:roman_IV}, it corresponds to Gevrey regularity when $0 < \alpha < 1$, while also encompassing analyticity and even stronger super-analyticity when $\alpha \geqslant 1$ (at least for the upper bound estimates).
	\end{remark}

	Although the results presented above are established for weighting functions $w \in \mathcal{W}_\beta$, the novel methods developed in this paper can, with suitable modifications, be extended to more general weighting functions, including the one studied in \cite[Theorem 1.3]{TL26a}. To avoid undue length, we defer a detailed investigation of such extensions to a separate paper, as also previously indicated in \cite{TL26a}.

	We conclude this section by emphasizing that the results formulated for nonresonant rotations are directly applicable to general dynamical systems that are smoothly conjugated to such rotations. Figure \ref{FIG1} illustrates the conjugacy $T(h(\theta))=h({\mathscr{T}}_\rho(\theta))$ for $\theta\in\mathbb{T}^d$ and an irrational vector $ \rho \in \mathbb{R}^d $, where $T\colon X \to X$ and $h\colon\mathbb{T}^d \to X$ are smooth diffeomorphisms. Typically, such a conjugacy is established via certain conjugacy theorems---the celebrated KAM theorem being a prime example. The result applies to an observable $f$ on $X$ when its pullback $f\circ h$ belongs to the required Fourier class; the spatial average is taken with respect to $h_*(d\theta)$. In particular, \cite[Theorem 4.1]{TL26b} relies on novel KAM techniques to address the logarithmic $C^\infty$ and $\alpha$-Gevrey regularities, corresponding to the situations considered in Theorem \ref{OPTT2}\footnote{However, the regularity of the conjugacy therein requires further investigation before it can be applied to the results of the present paper.}. Consequently, our results enjoy a broad range of applications.

	\begin{figure}[htbp]
		\centering %
		\begin{adjustbox}{scale=0.75}
			\begin{minipage}[c]{0.60\textwidth} %
				\centering %
				\begin{tikzpicture}[scale=1.4] %
					\draw[->,  thick, -{Latex[scale=1.1]}] (-0.5cm,0cm) -- (5cm,0cm) node[right]{\(x\)}; 
					\draw[->,  thick, -{Latex[scale=1.1]}] (0cm,-0.5cm) -- (0cm,4.5cm) node[above]{\(y\)}; 
					
					\fill (0cm,0cm) circle (1.5pt) node[below left]{\(O\)};

					\draw[fill=white, semithick, decoration={markings,
						mark=between positions 0.1 and 0.9 step 0.1 with {\arrow{latex}},
					},postaction={decorate}
					] plot[smooth, tension=0.5] coordinates {
						(3cm+0.3cm,2.4cm+0.3cm) 
						(2.5cm+0.3cm,3cm+0.3cm) 
						(1.7cm+0.3cm,2.7cm+0.3cm) 
						(1cm+0.3cm,3cm+0.3cm) 
						(0.5cm+0.3cm,2.5cm+0.3cm) 
						(0.3cm+0.3cm,1.5cm+0.3cm) 
						(1cm+0.3cm,1.3cm+0.3cm) 
						(1.5cm+0.3cm,0.5cm+0.3cm) 
						(2.5cm+0.3cm,0.6cm+0.3cm) 
						(3cm+0.3cm,0.8cm+0.3cm) 
						(3.8cm+0.3cm,1.5cm+0.3cm) 
						(3.5cm+0.3cm,2.5cm+0.3cm)
						(3.3cm+0.3cm,2.6cm+0.3cm)
					} -- cycle;

					\fill[black!90] (3cm+0.3cm,2.4cm+0.3cm) circle (1.2pt) node[above right, yshift=0.2cm, xshift=-0.2cm, black] {\(h(\theta_0)\)};
					\fill[black]  (1cm+0.3cm,3cm+0.3cm) circle (1.2pt) node[above,black] {\(h(\theta_0+\rho)\)};

					\draw[dashed, ->, black, semithick] (2cm+0.3cm,1.8cm+0.3cm) -- (4.0cm+0.3cm,1.8cm+0.3cm) node[below] {\(\theta\)};
					\draw[dashed, ->, black, semithick] (2cm+0.3cm,1.8cm+0.3cm) -- (2cm+0.3cm,3.8cm+0.3cm);
					\draw[dashed, <-, black, semithick] (0cm+0.3cm,1.8cm+0.3cm) -- (2.0cm+0.3cm,1.8cm+0.3cm);
					\draw[dashed, ->, black, semithick] (2cm+0.3cm,1.8cm+0.3cm) -- (2cm+0.3cm,-0.2cm+0.3cm);

					\draw[-latex,black!90,semithick] (2cm+0.3cm,1.8cm+0.3cm) -- (3cm+0.3cm,2.4cm+0.3cm);
					\draw[-latex,black,semithick] (2cm+0.3cm,1.8cm+0.3cm) -- (1cm+0.3cm,3cm+0.3cm);

					\tikzset{
						arc arrow/.style = {
							decoration = {
								markings,
								mark = at position 0.68 with {\arrow{latex}}
							},
							postaction = decorate
						}
					}

					\draw[black!90, semithick, arc arrow] (2.9cm+0.3cm,1.8cm+0.3cm) arc (0:40:0.69cm)
					node[midway, below right, yshift=0.25cm] {\scalebox{0.8}{\(\theta_0\)}};

					\tikzset{arc arrow/.style = {
							decoration = {markings,
								mark = at position 0.55 with {\arrow{latex}}
							},
							postaction = decorate
					}}

					\draw[black, semithick, arc arrow] (1.8cm+0.3cm,1.8cm+0.3cm) ++(0:1cm) arc (10:115:1cm)
					node[midway, below left, xshift=0.25cm] {\scalebox{0.8}{\(\theta_0+\rho\)}};
					
					\fill (2cm+0.3cm,1.8cm+0.3cm) circle (1.5pt) node[below left] {$O'$};
				\end{tikzpicture}
			\end{minipage}
			\hfill %
			\begin{minipage}[c]{0.60\textwidth} %
				\centering %
				\begin{tikzcd}[row sep=3.9 cm, column sep=3.9 cm, arrows={thick}]
					\mathbb{T}^d  \arrow[r, " {\mathscr{T}}_\rho"{font=\large}] \arrow[d, "h"'{font=\large}] & \mathbb{T}^d  \arrow[d, "h"{font=\large}] \\
					X \arrow[r, "T"'{font=\large}] & X
				\end{tikzcd}
			\end{minipage}
		\end{adjustbox}
		\caption{Geometric properties (for $d=1$) and the commutative diagram.}\label{FIG1}
	\end{figure}
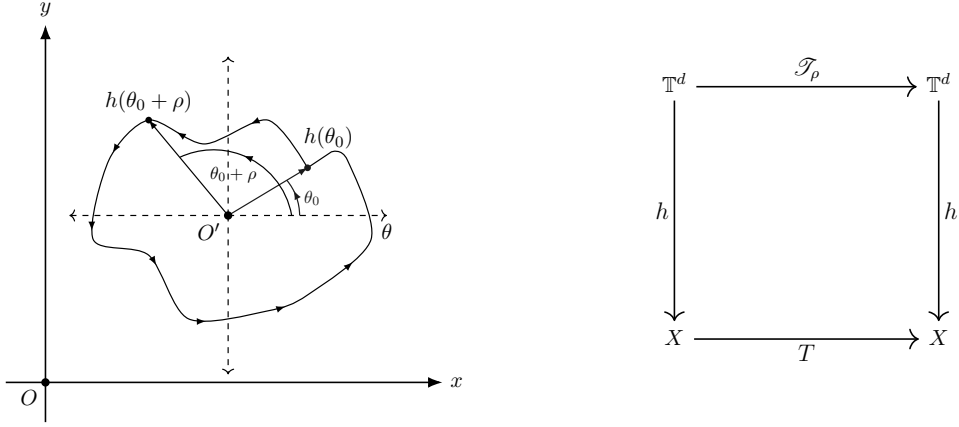

	\subsection{On the optimality of Theorem \ref{OPTT2}: Multiple aspects}\label{SECOPT}

	The present section is devoted to a detailed elaboration on the optimality of Theorem \ref{OPTT2}. The optimality manifests itself in \textit{twenty-one} aspects, including rotations, initial points, regularity assumptions, and the comparison of indices for the upper and lower bounds of the convergence rates, among others. It is worth noting that the optimality discussed here admits both strong and weak forms, and thus cannot be treated uniformly.

	\begin{itemize}
		
		\item Optimality in all Cases \ref{item3:roman_I}--\ref{item3:roman_IV}: 
		
		\begin{enumerate}[label=(\alph*)] 
			\item Optimality is attained for almost every rotation vector $\rho$ with respect to the Lebesgue measure (or in the measure-theoretic sense).
			
			\item \label{opitem:b} Moreover, when $l=2$ (i.e., $f \in \mathbb{R}_{\widetilde{\Delta}}^2 (\cdot ) $), according to Case \ref{item:roman_II} of Theorem \ref{OPTT1}, a uniform sense of optimality holds in terms of the initial value $\theta$. In other words, both the upper and lower bound estimates hold uniformly in $\theta$, which is noteworthy.
			
			\item In a slightly weaker sense, according to Case \ref{item:roman_I} of Theorem \ref{OPTT1} and Corollary \ref{OPTCORO1}, optimality manifests differently depending on the initial value $\theta$ in the discrete case for $d\geqslant1$ and in the continuous case for $d\geqslant2$: the upper bound holds uniformly for all initial values, whereas one scalar observable gives the lower bound at every point of a dense set, along times that may depend on that point.
			
			\item 
			Qualitative upper bounds on the convergence rate for the present regularity classes were also investigated in \cite{TL24a,TL24b}. To ensure rapid convergence, certain integrability and smallness conditions were imposed therein. Theorem \ref{OPTT2} precisely demonstrates that these conditions are optimal in the sense of guaranteeing rapid convergence, implying that fast convergence is not unconditionally attainable.
		\end{enumerate}

		\item Optimality in Case \ref{item3:roman_I}--\ref{TH3.3-I-1}: 
		
		\begin{enumerate}[label=(\alph*),start=5]
			\item \label{opitem:e} When $d=1$, the condition $\ell>1$ for the upper bound is optimal, since for $\ell \leqslant 1$,  we have
			\[\sum\limits_{k \in {\mathbb{Z}} \setminus \left\{ 0 \right\}} {\widetilde \Delta {{\left( {{{\left| k \right|}}} \right)}^{ - 1}}}    \gtrsim    \sum\limits_{j =1}^{\infty} {{\frac{1}{{  j }^{  \ell }}}}  =  + \infty  ,\]
			which may result in the observables not being well-defined.
			
			\item  The convergence rate achieves optimality for $d=1$, as the upper bound $\mathcal{O}(N^{-\ell'})$ and the lower bound $\mathcal{O}(N^{-\ell})$ closely match due to the condition $1<\ell'<\ell$. 
			
			\item 
			For $d \geqslant 2$, optimality is also manifested in the convergence rate in a slightly weaker sense. Specifically, for sufficiently large $\ell$, we have $1<\ell'<d^{-1}\ell$, implying that the indices $\ell'$ and $\ell$  can satisfy $\ell'=\mathcal{O}^{\#}(\ell)$; or more precisely, $\ell' $ can be taken to arbitrarily approach $d^{-1}\ell$.

			\item \label{opitem:h} 
		Furthermore, a distinct aspect of optimality concerns the regularity of the weighting function $w\in C_0^M([0,1])$. Specifically, the subsequential lower bound of order $N^{-\ell}$ in Case \ref{item3:roman_I}--\ref{TH3.3-I-1} remains valid for every such weight with $M\geqslant\max\{\ell,2\}$, including $w\in C_0^\infty([0,1])$. Thus, employing a weighting function with higher regularity cannot, in general, yield convergence faster than this lower bound. This demonstrates that the regularity of the observable is sometimes decisive for rapid convergence.
				
				The importance of endpoint flatness is illustrated by \cite[Section~5]{TL24b}, where certain weights with only finitely many vanishing endpoint derivatives give convergence of finite polynomial order even for trigonometric observables in continuous multiple-weighted averages. These examples demonstrate that finite smoothness and finite endpoint flatness alone are insufficient to ensure convergence of arbitrary polynomial order.
		\end{enumerate} 
		
		\item  Optimality in Case \ref{item3:roman_I}--\ref{TH3.3-I-2}: 
		
		\begin{enumerate}[label=(\alph*),start=9]
			
			\item \label{opitem:i} Analogous to Item \ref{opitem:e}, the optimality is manifested in the condition $\ell>d\geqslant1$ for the upper bound; indeed, for $\ell \leqslant d$, we obtain
			\[\sum\limits_{k \in {\mathbb{Z}^d} \setminus \left\{ 0 \right\}} {\widetilde \Delta {{\left( {{{\left\| k \right\|}_{{\ell ^\infty }}}} \right)}^{ - 1}}}  \gtrsim \sum\limits_{j = 1}^\infty  {\frac{1}{{{j^{\ell  - d + 1}}}}}  =  + \infty .\]

			\item \label{opitem:j} The convergence rate achieves optimality for $d=2$, as the upper bound $\mathcal{O}(T^{-\ell'})$ and the lower bound $\mathcal{O}(T^{-\ell})$ closely match due to the condition $1<\ell'<(d-1)^{-1}\ell=\ell$.
			
			\item For $d \geqslant 2$, optimality is additionally manifested in the convergence rate in a slightly weaker sense. Specifically, for sufficiently large $\ell$, the indices $\ell'$ and $\ell$ in Item \ref{opitem:j} can satisfy $\ell'=\mathcal{O}^{\#}(\ell)$, or more precisely, $\ell' $ can be taken to arbitrarily approach $ (d-1)^{-1}\ell$.
			
			\item The dichotomy into the cases $d=1$ and $d \geqslant 2$ is also optimal, as the convergence rates possess fundamentally different asymptotic structures.
			
			\item 
			Finally, when $ d = 1 $, relaxing the condition on $ \rho $ to $ \rho \ne 0 $ is also optimal. For $\rho=0$, the weighted average equals $f(\theta)$ for every $T$, which need not equal the spatial average.
		\end{enumerate}

		\item   Optimality in Case \ref{item3:roman_I}--\ref{TH3.3-I-3}:
		\begin{enumerate}[label=(\alph*),start=14]
			\item   The optimality of the setting $\ell \geqslant 2$ for the upper bound parallels that in Item \ref{opitem:i}, yet in a stronger sense due to  the presence  of a logarithmic factor.
			
			\item
			In this case, the optimality of the convergence rate is stronger than that in Case \ref{item3:roman_I}--\ref{TH3.3-I-2}, since the polynomial exponents of the upper bound $ \mathcal{O} (T^{-\ell} ) $ and the lower bound $ \mathcal{O}({T^{ - \ell }}{\left( {\log T} \right)^{ - \eta }}) $ coincide, differing overall only by a logarithmic factor $ \mathcal{O}({\left( {\log T} \right)^{ - \eta }}) $.  In conjunction with Item \ref{opitem:b} concerning $l=2$, the upper bound and the subsequential lower bound yield the following limit-superior statement: for almost every $ \rho \in \mathbb{R}^2 $, there exists a family of observables $ f \in   \mathbb{R}^2_{\widetilde{\Delta}}(\mathbb{T}^2) $ with $\widetilde{\Delta}$ given in Case \ref{item3:roman_I}--\ref{TH3.3-I-3}, such that\footnote{For other cases concerning the optimality of the convergence rate, analogous limit-superior or limit-inferior statements can be established; for the sake of brevity, we omit the details.}
			\[ \mathop {\limsup }\limits_{T \to  + \infty } \frac{1}{{ \log T}}\log \Big(\mathbf{Error}_T^{\mathrm{con}}\left(w, f,\rho,\theta\right)\Big) = -\ell, \quad \forall \theta \in \mathbb{T}^2.\]
			Here and below, $\log 0$ is understood as $-\infty$.

		\end{enumerate}

		\item  Optimality in Case \ref{item3:roman_II}:
		
		\begin{enumerate}[label=(\alph*),start=16]
			\item In the discrete setting with $d \geqslant 1$ and the continuous setting with $d \geqslant 2$, the optimality can be observed in the upper bounds $\mathcal{O}(N^{-m})$ and $\mathcal{O}(T^{-m})$ (for an arbitrary $m$), alongside the subsequential lower bounds of orders $\mathcal{O}(\widetilde{\Delta}(N)^{-1})$ and $\mathcal{O}(\widetilde{\Delta}(T)^{-1})$, under the additional assumption $\widetilde\Delta(x)\leqslant\exp(x^\zeta)$ for some $0<\zeta<\beta^{-1}$ and all sufficiently large $x$. Indeed, in such scenarios, the asymptotic behavior of $\widetilde\Delta(x)$ can closely approximate that of an arbitrary polynomial type. A concrete example of this, as discussed at the end of Section \ref{SUBSEC21}, is $\widetilde{\Delta}(x) \sim \exp((\log x)(\log \log x)^\sigma)$ with $\sigma>0$.

			\item An optimality argument similar to that in Item \ref{opitem:h} remains valid in this context, under the growth assumption on $\widetilde\Delta$ stated in the preceding item. Given that $C^\infty$ observables represent the most prevalent class in physical numerical simulations, this optimality is of fundamental importance.
		\end{enumerate}

		\item  Optimality in Case \ref{item3:roman_III}:
		\begin{enumerate}[label=(\alph*),start=18]
			\item In the discrete setting with $d \geqslant 1$ and the continuous setting with $d \geqslant 2$, the optimality is demonstrated by the fact that the logarithmic exponents of the upper bounds  $\mathcal{O}\bigl(\exp(-c(\log N)^{\lambda})\bigr)$ and $\mathcal{O}\bigl(\exp(-c(\log T)^{\lambda})\bigr)$ (for some $ 0<c<1 $) and those of the lower bounds $\mathcal{O}\bigl(\exp(-(\log N)^\lambda)\bigr)$ and $\mathcal{O}\bigl(\exp(-(\log T)^\lambda)\bigr)$ coincide. Combined with Item \ref{opitem:b} for $l=2$, the upper bounds and the subsequential lower bounds yield the following limit-inferior statements. In each setting separately, for almost every $\rho\in\mathbb R^d$ (with $d$ as specified above), there exists a family of observables $f\in\mathbb R^2_{\widetilde\Delta}(\mathbb T^d)$, with $\widetilde\Delta$ given in Case \ref{item3:roman_III}, such that the corresponding identity below holds; the choices in the discrete and continuous settings need not coincide. In the discrete case,
			\[\mathop {\liminf }\limits_{N \to  + \infty } \frac{1}{{\log \log N}}\log \left( { - \log \Big({\mathbf{Error}}_N^{{\text{dis}}}\left( {w,f,\rho ,\theta } \right)\Big)} \right) = \lambda ,\quad \forall \theta  \in {\mathbb{T}^d},\]
			and, in the continuous case, with its separately chosen observables,
			\[\mathop {\liminf }\limits_{T \to  + \infty } \frac{1}{{\log \log T}}\log \left( { - \log \Big({\mathbf{Error}}_T^{{\text{con}}}\left( {w,f,\rho ,\theta } \right)\Big)} \right) = \lambda ,\quad \forall \theta  \in {\mathbb{T}^d}.\]
			At zero error, the displayed expression $\log(-\log(\mathbf{Error}))$ is understood as $+\infty$.
		\end{enumerate}

		\item  Optimality in Case \ref{item3:roman_IV}:
		
		\begin{enumerate}[label=(\alph*),start=19]
			\item \label{opitem:s} The convergence rate in the discrete setting achieves optimality for $d=1$ when $ \alpha \to 0^+ $, as the upper bound $ o\bigl(\exp(-N^\xi)\bigr) $ and the lower bound $\mathcal{O}\bigl(\exp(-N^{\alpha})\bigr)$ closely match due to the condition $ 0 < \xi  < \alpha {\left( {\alpha \beta  + d} \right)^{ - 1}} \sim \alpha  $.
			
			\item In both the discrete and continuous cases for $d \geqslant 2$, a slightly weaker optimality is manifested as $ \alpha \to 0^+ $. In this scenario, the indices $\xi$ and $  \alpha $ in Item \ref{opitem:s} can satisfy $\xi = \mathcal{O}^{\#}(\alpha)$; or more precisely, $\xi $ can be taken to arbitrarily approach $ d^{-1}\alpha$.
			
			\item Even as $\alpha \to (\beta^{-1})^-$, the optimality persists when $\beta \to +\infty$, in the discrete case for $d\geqslant1$ and the continuous case for $d\geqslant2$. In this scenario, upon taking the limit $\alpha \to (\beta^{-1})^-$,  the indices $\xi$ and $\beta$ in Item \ref{opitem:s} can satisfy $\xi = \mathcal{O}^{\#}(\beta^{-1})$; or more precisely, $\xi $ can be taken to arbitrarily approach $ (1+d)^{-1} \beta^{-1}$, as $\beta \to +\infty$.
		\end{enumerate} 
	\end{itemize}

	The selection of alternative weighting functions can also yield further acceleration of convergence; see \cite[Theorem 1.3]{TL26a}. Indeed, determining the fastest possible rate of convergence for such weighting functions remains a fundamental open problem; see \cite{PR26}. In a forthcoming paper, employing techniques distinct from those utilized in the present work, we shall provide a complete answer to this question for toral translations: \textit{the linear-exponential rate of convergence is generally unattainable; however, any prescribed slower rate can be achieved via the construction of an appropriate weighting function.}

	\section{Principal contributions and comparison with previous literature}\label{SECNOV}	
	
	As mentioned in Section \ref{SEC31}, apart from a small portion of Theorem \ref{OPTT2}---which consists of results previously established in the literature or variants thereof---\textit{all other results in the present paper constitute novel contributions, both in their conclusions and in the underlying techniques.} We elaborate on these in detail below.
	
	As a principal contribution, Theorem \ref{OPTT1} and Corollary \ref{OPTCORO1} establish the \textit{first} lower bound estimates for the convergence rates of (non-uniformly) weighted Birkhoff averages on the torus, via almost all rotations and specific (or even all) initial points, under abstract and general regularity assumptions. From a technical standpoint, our approach partially adapts the ideas employed for the unweighted setting in \cite{Yoc80,Yoc95,KLM21} among others, while simultaneously incorporating the quantitative techniques developed in \cite{TL25b,TL26a} for upper bound estimates in the weighted case. We remark that the approaches in the latter build upon \cite{DSSY17, DY18}, with the introduction of new technical ingredients. Interestingly, establishing these lower bounds essentially relies upon the corresponding upper bound estimates as a crucial step, thereby revealing a profound underlying connection between the two. 
	
	A substantial body of literature has been devoted to lower bound estimates for the rate of convergence of Birkhoff averages; see, for instance, \cite{Kre78,Yoc80,Yoc95,Ryz23,Ryz25,PR26}. In particular, for translations on the torus, certain intrinsic connections between regularity and nonresonance conditions regarding this rate were investigated in \cite{Yoc80,Yoc95}. Theorem \ref{OPTT1} and Corollary \ref{OPTCORO1} establish an analogous formulation in the weighted setting, albeit with potentially greater precision, as we simultaneously incorporate techniques from both harmonic analysis and number theory. 	In classical ergodic theory,   one   generally only obtains lower bounds on the Birkhoff average convergence rate for a fixed initial point \cite{KLM21}. In contrast,  our results in Theorem \ref{OPTT1} hold for all initial points in a dense set, or for all initial points on the torus, which marks a significant difference.
	
	In essence, Theorem \ref{OPTT1} and Corollary \ref{OPTCORO1} establish a fundamental observation: \textit{the rapid uniform convergence of weighted Birkhoff averages cannot, in general, be achieved if the observable lacks sufficient regularity.} Consequently, certain integrability or smallness conditions---such as those proposed in \cite{TL24a,TL24b} to guarantee uniform convergence at an arbitrary polynomial (or exponential) rate---are indeed \textit{indispensable} (see, for instance, \eqref{keji}). It should be emphasized that the phrase ``in general'' mentioned above is of critical importance, as we construct the desired observable for a fixed universal rotation (see Remark \ref{RE31}). As an intriguing contrast, it was recently demonstrated in \cite{Ton26} that for any given irrational rotation, there exists a family of observables with  low regularity (requiring only the absolute summability of their Fourier coefficients) for which the weighted Birkhoff averages exhibit a quantitative exponential upper bound on the rate of convergence (with the rate taking an identical form to that in \cite{TL25b,TL26a}). \textit{In other words, the present paper and \cite{Ton26} collectively provide an in-depth investigation into the effects of regularity and nonresonance conditions on the convergence rate of weighted Birkhoff averages.}

	We now turn our attention to Theorem \ref{OPTT2}. As an additional principal contribution, it establishes, for the \textit{first} time, \textit{explicit and optimal}  rates of convergence for weighted Birkhoff averages under various regularity settings. We refer the reader to Section \ref{SECOPT} for a detailed discussion on optimality. Theorem \ref{OPTT2} consists of two components: the lower bound estimates provided by Theorem \ref{OPTT1} and Corollary \ref{OPTCORO1}, and the upper bound estimates derived either from the novel techniques introduced in the present paper or as variants of established historical results. We shall now elaborate on the latter component in detail.  In Case \ref{item3:roman_I}--\ref{TH3.3-I-1}, the Fourier expansion and the estimates for the weighting function are based on the approaches in \cite[Theorem 3.1]{DSSY17}, \cite[Theorem 3.1]{DY18}, \cite[Theorem 3]{DM23}, and \cite[Theorem 2.1]{TL24a}. To obtain the stated range $\ell>d$ and $1<\ell'<d^{-1}\ell$, we combine these estimates with pairwise separation of the residues $\langle k,\rho\rangle\pmod{\mathbb Z}$ in bounded frequency boxes and a dyadic decomposition of the frequency lattice. This provides a stronger estimate than those known in higher dimensions, an improvement that has remained out of reach until now. We also present an alternative proof for $d=1$ based on continued fractions and the Denjoy--Koksma inequality. In Case \ref{item3:roman_I}--\ref{TH3.3-I-2}, the upper bound for $d=1$ provides a quantitative version of \cite[Theorem 3.4]{TL24a} by utilizing the techniques from \cite{TL25b,TL26a}. For $d\geqslant2$, the proof uses the methods of \cite{Sal04} to estimate continuous-time small divisors. Under the assumption $\ell>d$, the different nonresonance exponents yield $1<\ell'<d^{-1}\ell$ in the discrete case for $d\geqslant1$ and $1<\ell'<(d-1)^{-1}\ell$ in the continuous case for $d\geqslant2$. Case \ref{item3:roman_I}--\ref{TH3.3-I-3} builds upon Case \ref{item3:roman_I}--\ref{TH3.3-I-2} by investigating the problem under more critical regularity assumptions. The upper bound estimates in Case \ref{item3:roman_II} are direct applications or variants of the results from \cite[Theorem 3.1]{DSSY17}, \cite[Theorem 1.1]{DY18}, \cite[Theorem 3]{DM23}, and \cite[Theorem 2.1]{TL24a}. Finally, the upper bounds in Cases \ref{item3:roman_III} and \ref{item3:roman_IV} do not require the previously incorporated number-theoretic tools; instead, by relying on the techniques from \cite{TL25b, TL26a}, they establish quantitative results for logarithmic $C^\infty$ regularity and Gevrey regularity for the \textit{first} time.
	
	We emphasize that historical methods fail to achieve the optimal convergence rates established in Theorem \ref{OPTT2}. For instance, in Case \ref{item3:roman_I}--\ref{TH3.3-I-1} for $d=1$, relying on the techniques from \cite{DSSY17,DY18,TL24a} restricts one to sufficiently large $\ell$ and yields a slower convergence rate, as conditions such as \eqref{keji} must be imposed. This differs markedly from our main results. Therefore, it is indeed \textit{essential} to develop the aforementioned refined methods to handle the small divisors arising from universal nonresonance conditions.
	
	In conclusion, the optimal results and  techniques encompassed by Theorem \ref{OPTT2} are, for the most part,  \textit{novel}, with some of the newly introduced methods potentially making their first appearance in weighted ergodic theory\footnote{For instance, the techniques from \cite{Sal04} initially focused on   KAM theory; see Section \ref{PFTHI} for details.}. Therefore, Theorem \ref{OPTT2} could potentially act as a useful link, offering new perspectives on related problems in both harmonic analysis and number theory.

	\section{Proofs of the main  results}\label{SECPRO}
	
	The present section is devoted to the detailed proofs of Theorem \ref{OPTT1} (Main Theorem I), Corollary \ref{OPTCORO1}, and Theorem \ref{OPTT2} (Main Theorem II).

	\subsection{Proof of Theorem \ref{OPTT1}: Lower bounds for uniform convergence rates}\label{PFT1}
	
	\begin{proof}	
		We first address the \textit{discrete} case in Theorem \ref{OPTT1}. It suffices to consider rotations on the torus  $ \mathbb{T}  $. To begin, let us recall a basic fact: for \textit{almost every} $ \rho  \in {\mathbb{T}  \setminus \mathbb{Q}} $ with approximants $ {\left\{ {{p_n}/{q_n}} \right\}_{n \in {\mathbb{N}^ + }}} $, there exist infinitely many $ \nu \in \mathbb{N}^+ $ such that 
		\begin{equation*}
			{q_{\nu + 1}} > \nu{q_\nu}
		\end{equation*} 
		by virtue of the divergence of the series $ \sum\nolimits_{\nu = 1}^\infty  {{\nu^{ - 1}}}  =  + \infty  $ and the Borel-Bernstein theorem (cf. \cite[Theorem 3.2.5]{QQ13}).

		For the given approximation function $\widetilde{\Delta} $, we \textit{inductively} choose a strictly increasing sequence of positive integers $\{m_s\}_{s \in \mathbb{N}^+}$ satisfying: 
		\begin{enumerate}[label=(\alph*)] 
			\item \label{item:a} $ {q_{m_s + 1}} > m_s{q_{m_s}}, \quad \forall s \in \mathbb{N}^+$;
			\item \label{item:b} $ \mathop {\lim }\limits_{s \to \infty } \left| {{m_{s + 1}} - {m_s}} \right| =  + \infty  $;
			\item \label{item:c} $ \sum\limits_{j = 1}^{s- 1} {\frac{1}{{\widetilde \Delta \left( {{q_{{m_j}}}} \right)}}\exp \left( { - C_1{{\left( {\frac{{{q_{{m_s}}}}}{{{q_{{m_j} + 1}}}}} \right)}^{\beta^{-1}}}} \right)}  \ll \frac{1}{{\widetilde \Delta \left( {{q_{{m_s}}}} \right)}}, \quad 0<C_1<e^{-1}\beta\lambda^{-1/\beta}, \quad  \forall s \gg 1 $.
		\end{enumerate}
		
		Here, we recall that $\beta>1$ is the constant such that the weighting function $w  \in \mathcal{W}_\beta$. It should be noted that these requirements are indeed achievable, given that $x^\varepsilon \leqslant \widetilde{\Delta}(x) \leqslant \exp(x^\zeta)$ for some $\varepsilon > 0$ and $0 < \zeta < \beta^{-1}$, and that $q_n$ grows at least exponentially.
		Passing to a further subsequence if necessary, we may also require
		\[
		\widetilde\Delta(q_{m_{s+1}})\geqslant16\widetilde\Delta(q_{m_s}),\quad s\geqslant1.
		\]
		This is readily achieved by virtue of assumption \ref{item:a} and the hypothesis on $\widetilde \Delta$.

		We shall establish  Cases \ref{item:roman_I}, \ref{item:roman_II}, and \ref{item:roman_III} in sequence. Our analysis focuses primarily on Case \ref{item:roman_I}, as Cases \ref{item:roman_II} and \ref{item:roman_III} can be treated as variants thereof.
		\vspace{3mm}

		\noindent \textbf{\textit{Proof of Case  \ref{item:roman_I}:}} \quad 	With these preparations, we are now equipped to construct the desired observable $\Psi \in \mathbb{R}_{\widetilde{\Delta}}(\mathbb{T} )$ such that the weighted Birkhoff average admits a lower bound for the convergence rate of order $ \mathcal{O}(\widetilde \Delta {\left( N \right)^{ - 1}}) $. Define 
		\[\Psi \left( \theta  \right) = \operatorname{Re} \sum\limits_{j = 1}^\infty  {\frac{{{\exp\left(2\pi i{q_{{m_j}}}\theta \right)}}}{{\widetilde \Delta \left( {{q_{{m_j}}}} \right)}}} ,\quad \theta \in \mathbb{T}.\]
		It is evident that $ \Psi $ is well-defined and has zero mean on $ \mathbb{T}  $, i.e., $ \int_{{\mathbb{T} }} {\Psi \left( x  \right) {d}x }  = 0 $. The geometric condition gives absolute summability of its actual coefficients, so $\Psi$ belongs to the class in Definition \ref{DEFBANACH}.
		
		Construct the set  $\Xi:= {\left\{ {{\{{  {\iota\rho }\} }}} \right\}_{\iota \in {\mathbb{N}  }}} $. Then $ \Xi $ is dense in $ \mathbb{T}  $ as a consequence of the irrationality of $ \rho $ and Kronecker's approximation theorem. For $ s \in \mathbb{N}^+ $ sufficiently large, consider the following weighted Birkhoff average with an initial point $ \theta \in \Xi $ (i.e., there exists some $ \iota \in \mathbb{N}  $ such that $ \theta =\{\iota\rho\} $):
		\begin{align}
			&\;\frac{1}{{{A_{{q_{{m_s}}}}}}}\sum\limits_{n = 0}^{{q_{{m_s}}} - 1} {w\left( {n/{q_{{m_s}}}} \right)\Psi \left( {\mathscr{T}_\rho ^n\left( \theta \right)} \right)}-\int_{{\mathbb{T} }} {\Psi \left( x  \right) {d}x }\notag \\
			= &\;\operatorname{Re} \left( {\sum\limits_{j = 1}^\infty  {\frac{1}{{{A_{{q_{{m_s}}}}}}}\left( {\sum\limits_{n = 0}^{{q_{{m_s}}} - 1} {w\left( {n/{q_{{m_s}}}} \right)\frac{{{\exp\left(2\pi i{q_{{m_j}}}{(\theta +n\rho)} \right)}}}{{\widetilde \Delta \left( {{q_{{m_j}}}} \right)}}} } \right)} } \right)\notag\\
			\label{hebing}: = &\;{\mathcal{S}_1}\left( {{q_{{m_s}}},\theta} \right) + {\mathcal{S}_2}\left( {{q_{{m_s}}},\theta} \right) + {\mathcal{S}_3}\left( {{q_{{m_s}}},\theta} \right),
		\end{align}
		where the terms  ${\mathcal{S}_1}\left( {{q_{{m_s}}},\theta} \right)$, $ {\mathcal{S}_2}\left( {{q_{{m_s}}},\theta} \right)$, and ${\mathcal{S}_3}\left( {{q_{{m_s}}},\theta} \right)$ are respectively defined as:
		\begin{align*}
			{\mathcal{S}_1}\left( {{q_{{m_s}}},\theta} \right) &:= \operatorname{Re} \left( {\frac{1}{{{A_{{q_{{m_s}}}}}}}\sum\limits_{j = 1}^{s - 1} {\sum\limits_{n = 0}^{{q_{{m_s}}} - 1} {w\left( {n/{q_{{m_s}}}} \right)\frac{{{\exp\left(2\pi i{q_{{m_j}}}{(\theta +n\rho)}\right)}}}{{\widetilde \Delta \left( {{q_{{m_j}}}} \right)}}} } } \right),\\
			{\mathcal{S}_2}\left( {{q_{{m_s}}},\theta} \right) &:= \operatorname{Re} \left( {\frac{1}{{{A_{{q_{{m_s}}}}}}}\sum\limits_{n = 0}^{{q_{{m_s}}} - 1} {w\left( {n/{q_{{m_s}}}} \right)\frac{{{\exp\left(2\pi i{q_{{m_s}}}{(\theta +n\rho)}\right)}}}{{\widetilde \Delta \left( {{q_{{m_s}}}} \right)}}} } \right),\\
			{\mathcal{S}_3}\left( {{q_{{m_s}}},\theta} \right) &:= \operatorname{Re} \left( {\frac{1}{{{A_{{q_{{m_s}}}}}}}\sum\limits_{j = s + 1}^\infty  {\sum\limits_{n = 0}^{{q_{{m_s}}} - 1} {w\left( {n/{q_{{m_s}}}} \right)\frac{{{\exp\left(2\pi i{q_{{m_j}}}{(\theta +n\rho)}\right)}}}{{\widetilde \Delta \left( {{q_{{m_j}}}} \right)}}} } } \right).
		\end{align*}
		We shall demonstrate that, given our choice of the sequence $\{m_s\}_{s \in \mathbb{N}^+}$, the term $\mathcal{S}_2(q_{m_s}, \theta)$ acts as the \textit{dominant} term, effectively controlling the magnitudes of $\mathcal{S}_1(q_{m_s}, \theta)$ and $\mathcal{S}_3(q_{m_s}, \theta)$. These terms will be estimated separately using \textit{different} techniques.
		
		The \textit{principal} difficulty lies in estimating the first term $\mathcal{S}_1(q_{m_s},\theta)$. We begin with a small-divisor analysis. For any integer $ M \geqslant 2 $, the following bound holds for $1 \leqslant j \leqslant s - 1$:
		\begin{equation}\label{divisor}
			\sum\limits_{n \in \mathbb Z } {\frac{1}{{{{\left| {{q_{{m_j}}}\rho  - n} \right|}^M}}}}  \leqslant {2^{M + 1}}q_{{m_j} + 1}^M.
		\end{equation}
		Indeed, since $p_{m_j}$ is the nearest integer to $q_{m_j}\rho$, we have
		\begin{align*}
			\sum\limits_{n \in \mathbb Z }{\frac{1}{{{{\left| {{q_{{m_j}}}\rho  - n} \right|}^M}}}}  &= \frac{1}{{{{\left| {{q_{{m_j}}}\rho  - {p_{{m_j}}}} \right|}^M}}} + \sum\limits_{n \ne {p_{{m_j}}}} {\frac{1}{{{{\left| {{q_{{m_j}}}\rho  - n} \right|}^M}}}} \\
			& \leqslant {\left( {2{q_{{m_j} + 1}}} \right)^M} + 2\sum\limits_{n = 1}^\infty  {\frac{1}{{{n^M}}}}+2^M \\
			& \leqslant {2^M}q_{{m_j} + 1}^M + {3^{ - 1}}{\pi ^2}+2^M\\
			& \leqslant {2^{M + 1}}q_{{m_j} + 1}^M.
		\end{align*}
		Here we have utilized \eqref{liangce} and the fact that $\sum_{n=1}^\infty n^{-M} \leqslant \sum_{n=1}^\infty n^{-2} = 6^{-1}\pi^2 $ for $M \geqslant 2$.   Recall that $w \in \mathcal{W}_\beta$ (see Definition \ref{generalwei}). Then, by integration by parts, it is evident that for any $\eta \ne 0$ and $\nu \in \mathbb{N}^+$,
		\begin{align}
			\left| {\int_0^1 {{{  w} }\left( z \right){\exp\left({2\pi i  \eta   z}\right)}dz} } \right| &= \left| {{{\left( {2\pi i \eta } \right)}^{ - \nu }}\int_0^1 {{D^\nu }{{  w} }\left( z \right){\exp\left({2\pi i  \eta  z}\right)}dz} } \right|\notag \\
			&\leqslant {\left| {2\pi  \eta } \right|^{ - \nu }}{\left\| {{D^\nu }{{  w} }} \right\|_{{L^1}\left( {0,1} \right)}}. \label{FBJF}
		\end{align}
		Utilizing \eqref{divisor} and \eqref{FBJF}, we obtain:
		\begin{align}
			\left| {\sum\limits_{n \in \mathbb Z }  {\int_0^1 {w\left( z \right){\exp\left({2\pi i{q_{{m_s}}}\left( {{q_{{m_j}}}\rho  - n} \right)z}\right)} {d}z} } } \right| 
			&	\leqslant  \sum\limits_{n \in \mathbb Z }  {\frac{{{{\left\| {{D^{M_j}w}} \right\|}_{{L^1}\left( {0,1} \right)}}}}{{{{\left| {2\pi {q_{{m_s}}}\left( {{q_{{m_j}}}\rho  - n} \right)} \right|}^{{M_j}}}}}} \notag \\
			&	\leqslant  \bar C\sum\limits_{n   \in \mathbb Z }  {\frac{{\lambda ^{{M_j}}}{M_j^{{M_j}\beta }}}{{{{\left| {2\pi {q_{{m_s}}}\left( {{q_{{m_j}}}\rho  - n} \right)} \right|}^{{M_j}}}}}} \notag \\
			&	\leqslant  \bar C\frac{{\lambda ^{{M_j}}}{M_j^{{M_j}\beta }}}{{{{\left( {2\pi {q_{{m_s}}}} \right)}^{{M_j}}}}} \cdot {2^{{M_j}+ 1}}q_{{m_j} + 1}^{{M_j}}\notag \\
			\label{S11}& \leqslant  \bar C{\lambda ^{{M_j}}}M_j^{{M_j}\beta }q_{{m_s}}^{ - {M_j}}q_{{m_j} + 1}^{M_j},  \quad  1\leqslant j \leqslant s-1,
		\end{align}
		for some integer $   M_j\geqslant 2 $ which will be determined later. Let $\bar C\exp(\mathscr{F}(M_j))$ denote the right-hand side of \eqref{S11}, where
		\[\mathscr{F}\left( x \right) := \beta x\log x+x\log \lambda + x\log {q_{{m_j} + 1}} - x\log {q_{{m_s}}}, \quad x \geqslant 2.\]
		A standard monotonicity argument allows us to choose $M_j$ for $1\leqslant j\leqslant s-1$ to minimize $\exp(\mathscr{F}(M_j))$. Specifically, setting
		\begin{equation}\label{MJ}
			{M_j} \sim {e^{ - 1}}{\left({\lambda ^{ - 1}} {{q_{{m_s}}}q_{{m_j} + 1}^{ - 1}} \right)^{{\beta ^{ - 1}}}}\geqslant 2, \quad 1 \leqslant j \leqslant s - 1
		\end{equation}
		yields  
		\[\exp \left( {\mathscr{F}\left( {{M_j}} \right)} \right)   \lesssim  \exp \left( { - C_1 {{\left( {{q_{{m_s}}}q_{{m_j} + 1}^{ - 1}} \right)}^{{\beta ^{ - 1}}}}} \right),\quad 1 \leqslant j \leqslant s - 1.\]
		In view of \eqref{S11}, this implies:
		\begin{equation}\label{S120}
			\left| {\sum\limits_{n \in \mathbb Z }  {\int_0^1 {w\left( z \right){\exp\left({2\pi i{q_{{m_s}}}\left( {{q_{{m_j}}}\rho  - n} \right)z}\right)} {d}z} } } \right| \lesssim\exp \left( { - C_1 {{\left( {{q_{{m_s}}}q_{{m_j} + 1}^{ - 1}} \right)}^{{\beta ^{ - 1}}}}} \right), \quad 1 \leqslant j \leqslant s - 1.
		\end{equation}
		Note that in \eqref{MJ},  $ M_j \to +\infty $ as $ s \to +\infty$  for all $ 1 \leqslant j \leqslant s-1 $ due to    assumption \ref{item:b} (because $ \left| {{m_s} - ({m_{s - 1}} + 1}) \right| \gg 1 $ for $ s $ sufficiently large and the sequence $ \{q_n\}_{n \in \mathbb{N}^+} $ exhibits  the exponential properties), and therefore our choice is indeed  valid. Next, by straightforward verification, we obtain
		\begin{align*}
			\left| {{\mathcal{S}_1}\left( {{q_{{m_s}}},\theta} \right)} \right| & \leqslant \left| {\frac{1}{{{A_{{q_{{m_s}}}}}}}\sum\limits_{j = 1}^{s - 1} {\sum\limits_{n = 0}^{{q_{{m_s}}} - 1} {w\left( {n/{q_{{m_s}}}} \right)\frac{{{\exp\left({2\pi i{q_{{m_j}}}{(\theta +n\rho)} }\right)}}}{{\widetilde \Delta \left( {{q_{{m_j}}}} \right)}}} } } \right|\notag \\
			&  \leqslant \sum\limits_{j = 1}^{s - 1} {\frac{1}{{\widetilde \Delta \left( {{q_{{m_j}}}} \right)}}\left| {\frac{1}{{{A_{{q_{{m_s}}}}}}}\sum\limits_{n = 0}^{{q_{{m_s}}} - 1} {w\left( {n/{q_{{m_s}}}} \right){\exp\left({2\pi i{q_{{m_j}}}n\rho }\right)}} } \right|} \notag \\
			&  = \sum\limits_{j = 1}^{s - 1} {\frac{1}{{\widetilde \Delta \left( {{q_{{m_j}}}} \right)}}\frac{1}{{{A_{{q_{{m_s}}}}}}}\left| {\sum\limits_{n \in \mathbb Z }  {w\left( {n/{q_{{m_s}}}} \right){\exp\left({2\pi i{q_{{m_j}}}n\rho }\right)}} } \right|}  .
		\end{align*}
		Applying the Poisson summation formula (cf. \cite[Chapter 3]{Gra14}) to the function 
		\[x\mapsto  {w} \left(x/q_{m_s}\right)\exp\left(2\pi i q_{m_j} x \rho\right),\]
		we deduce 
		\begin{align*}
			{\sum\limits_{n \in \mathbb Z }  {w\left( {n/{q_{{m_s}}}} \right){\exp\left({2\pi i{q_{{m_j}}}n\rho }\right)}} }   &=  {\sum\limits_{n \in \mathbb Z }  {\int_{\mathbb R } {w\left( {t/{q_{{m_s}}}} \right){\exp\left({2\pi i{q_{{m_j}}}\rho t}\right)} \cdot {\exp\left({ - 2\pi i n t}\right)} {d}t} } }  \\
			&=q_{m_s}  {\sum\limits_{n \in \mathbb Z }  {\int_0^1 {w\left( z \right){\exp\left({2\pi i{q_{{m_s}}}\left( {{q_{{m_j}}}\rho  - n} \right)z}\right)} {d}z} } }  .
		\end{align*}
		Consequently, we arrive at
		\begin{align}
			\left| {{\mathcal{S}_1}\left( {{q_{{m_s}}},\theta} \right)} \right| 
			& \leqslant \frac{{{q_{{m_s}}}}}{{{A_{{q_{{m_s}}}}}}}\sum\limits_{j = 1}^{s - 1} {\frac{1}{{\widetilde \Delta \left( {{q_{{m_j}}}} \right)}}\left| {\sum\limits_{n \in \mathbb Z }  {\int_0^1 {w\left( z \right){\exp\left({2\pi i{q_{{m_s}}}\left( {{q_{{m_j}}}\rho  - n} \right)z}\right)} {d}z} } } \right|} 	\notag \\
			\label{S18}&  \lesssim \sum\limits_{j = 1}^{s - 1} {\frac{1}{{\widetilde \Delta \left( {{q_{{m_j}}}} \right)}}\exp \left( { - C_1 {{\left( {{q_{{m_s}}}q_{{m_j} + 1}^{ - 1}} \right)}^{{\beta ^{ - 1}}}}} \right)}  \\
			\label{S19}& \leqslant \frac{1}{8{\widetilde \Delta \left( {{q_{{m_s}}}} \right)}}, \quad  s \gg 1.
		\end{align}
		Here, \eqref{S18} employs \eqref{S120} and the following uniform boundedness
		\[\frac{u}{{{A_u}}} = {\left( {\frac{1}{u}\sum\limits_{n = 0}^{u - 1} {w\left( {n/u} \right)} } \right)^{ - 1}} \to {\left( {\int_0^1 {w\left( x \right) dx} } \right)^{ - 1}} = 1 \quad \text{as} \quad u \gg 1,\]
		and \eqref{S19} uses assumption \ref{item:c}.

		We now turn our attention to the second term $\mathcal{S}_2(q_{m_s},\theta)$, which turns out to be the \textit{dominant} one. In view of  \eqref{zhishu}, \eqref{zhishu2} and $ \theta=\{\iota \rho\} \in \Xi $, we have
		\begin{align}
			&\; \left| {\frac{1}{{{A_{{q_{{m_s}}}}}}}\sum\limits_{n = 0}^{{q_{{m_s}}} - 1} {w\left( {n/{q_{{m_s}}}} \right){\exp \left({2\pi i{q_{{m_s}}}{(\theta +n\rho)} }\right)}}  - 1} \right|\notag \\
			= &\; \left| {\frac{1}{{{A_{{q_{{m_s}}}}}}}\sum\limits_{n = 0}^{{q_{{m_s}}} - 1} {w\left( {n/{q_{{m_s}}}} \right){\exp \left({2\pi i{q_{{m_s}}}{(\theta +n\rho)} }\right)}}  - \frac{1}{{{A_{{q_{{m_s}}}}}}}\sum\limits_{n = 0}^{{q_{{m_s}}} - 1} {w\left( {n/{q_{{m_s}}}} \right)} } \right|\notag \\
			\leqslant& \; \frac{1}{{{A_{{q_{{m_s}}}}}}}\sum\limits_{n = 0}^{{q_{{m_s}}} - 1} {w\left( {n/{q_{{m_s}}}} \right)\left| {{\exp\left({2\pi i{q_{{m_s}}}{(\theta +n\rho)} }\right)} - 1} \right|} \notag \\
			\leqslant&\; \frac{{{\widetilde C}}}{{{A_{{q_{{m_s}}}}}}}\sum\limits_{n = 0}^{{q_{{m_s}}} - 1} {w\left( {n/{q_{{m_s}}}} \right){{\left\| {{q_{{m_s}}}{(\theta +n\rho)} } \right\|}_\mathbb{Z}}} \notag \\
			=&\; \frac{{{\widetilde C}}}{{{A_{{q_{{m_s}}}}}}}\sum\limits_{n = 0}^{{q_{{m_s}}} - 1} {w\left( {n/{q_{{m_s}}}} \right){{\left\| {{q_{{m_s}}}{(\iota +n)\rho} } \right\|}_\mathbb{Z}}} \notag \\
			\leqslant&\; \frac{{\widetilde C^2}}{{{A_{{q_{{m_s}}}}}}}\sum\limits_{n = 0}^{{q_{{m_s}}} - 1} {w\left( {n/{q_{{m_s}}}} \right){(\iota+n)}{{\left\| {{q_{{m_s}}}\rho } \right\|}_\mathbb{Z}}}. \label{SSSS} 
		\end{align}
		Therefore, by \eqref{liangce}, \eqref{SSSS} and assumption \ref{item:a},  we can  further estimate
		\begin{align}
			\left| {\frac{1}{{{A_{{q_{{m_s}}}}}}}\sum\limits_{n = 0}^{{q_{{m_s}}} - 1} {w\left( {n/{q_{{m_s}}}} \right){\exp \left({2\pi i{q_{{m_s}}}{(\theta +n\rho)} }\right)}}  - 1} \right| &\leqslant \frac{{2\widetilde C^2}}{{{A_{{q_{{m_s}}}}}}}\sum\limits_{n = 0}^{{q_{{m_s}}} - 1} {w\left( {n/{q_{{m_s}}}} \right){q_{{m_s}}}q_{{m_s} + 1}^{ - 1}}  \notag \\
			\label{CCC}	&=  2\widetilde C^2{q_{{m_s}}}q_{{m_s} + 1}^{ - 1} \leqslant  2\widetilde C^2 m_s^{ - 1} 	\leqslant  \frac{1}{2},  
		\end{align}
		whenever $ s  $ is sufficiently large  (note that  $\iota \leqslant q_{m_s} $ in this case). This crucial fact directly implies that 
		\[\operatorname{Re} \left( {\frac{1}{{{A_{{q_{{m_s}}}}}}}\sum\limits_{n = 0}^{{q_{{m_s}}} } {w\left( {n/{q_{{m_s}}}} \right){\exp\left({2\pi i{q_{{m_s}}}(\theta+ n\rho) }\right)}}  } \right) \geqslant \frac{1}{2}, \quad s \gg 1,\]
		since $|z-1| \leqslant 2^{-1}$ implies $\operatorname{Re} z \geqslant2^{-1}$ for any $z \in \mathbb{C}$.
		Consequently, we arrive at the following key estimate:
		\begin{equation}\label{S2GUJI}
			\left| {{\mathcal{S}_2}\left( {{q_{{m_s}}},\theta} \right)} \right| = {\mathcal{S}_2}\left( {{q_{{m_s}}},\theta} \right) \geqslant \frac{1}{{2\widetilde \Delta \left( {{q_{{m_s}}}} \right)}}, \quad s \gg 1.
		\end{equation}
		
		For the third term $ {\mathcal{S}_3}\left( {{q_{{m_s}}},\theta} \right) $, the triangle inequality yields
		\begin{align}
			\left| {{\mathcal{S}_3}\left( {{q_{{m_s}}},\theta} \right)} \right| & \leqslant \left| {\frac{1}{{{A_{{q_{{m_s}}}}}}}\sum\limits_{j = s + 1}^\infty  {\sum\limits_{n = 0}^{{q_{{m_s}}} - 1} {w\left( {n/{q_{{m_s}}}} \right)\frac{{{\exp\left({2\pi i{q_{{m_j}}}(\theta+ n\rho) }\right)}}}{{\widetilde \Delta \left( {{q_{{m_j}}}} \right)}}} } } \right|\notag\\
			&  \leqslant \sum\limits_{j = s + 1}^\infty  {\frac{1}{{\widetilde \Delta \left( {{q_{{m_j}}}} \right)}}\left( {\frac{1}{{{A_{{q_{{m_s}}}}}}}\sum\limits_{n = 0}^{{q_{{m_s}}} - 1} {w\left( {n/{q_{{m_s}}}} \right)} } \right)}  = \sum\limits_{j = s + 1}^\infty  {\frac{1}{{\widetilde \Delta \left( {{q_{{m_j}}}} \right)}}} \notag \\ 
			\label{S31}&\leqslant \frac{1}{8\widetilde\Delta(q_{m_s})}, \quad s \gg 1.
		\end{align}
		Here \eqref{S31} follows from the geometric condition imposed when choosing the indices.
		
		Finally, substituting \eqref{S19}, \eqref{S2GUJI} and \eqref{S31} into \eqref{hebing}, we immediately derive a lower bound for the weighted Birkhoff average initialized at $\theta=\{\iota \rho\} \in \Xi$:
		\begin{align*}
			\left|\frac{1}{{{A_{{q_{{m_s}}}}}}}\sum\limits_{n = 0}^{{q_{{m_s}}} - 1} {w\left( {n/{q_{{m_s}}}} \right)\Psi \left( {\mathscr{T}_\rho ^n\left( \theta \right)} \right)}-\int_{{\mathbb{T} }} {\Psi \left( x  \right) {d}x } \right|
			&	\geqslant  \left| {{\mathcal{S}_2}\left( {{q_{{m_s}}},\theta} \right)} \right| - \left| {{\mathcal{S}_1}\left( {{q_{{m_s}}},\theta} \right)} \right| - \left| {{\mathcal{S}_3}\left( {{q_{{m_s}}},\theta} \right)} \right|\\
			&	\geqslant  \frac{1}{{2\widetilde \Delta \left( {{q_{{m_s}}}} \right)}} - \frac{1}{{8\widetilde \Delta \left( {{q_{{m_s}}}} \right)}} - \frac{1}{{8\widetilde \Delta \left( {{q_{{m_s}}}} \right)}}\\
			&	=  \frac{1}{{4\widetilde \Delta \left( {{q_{{m_s}}}} \right)}}, \quad s \gg 1.
		\end{align*}
		By setting $N=q_{m_s}$ for sufficiently large $s$, which yields infinitely many such $N$, we conclude the proof of Case \ref{item:roman_I}.
		\vspace{3mm}

		\noindent \textbf{\textit{Proof of Case  \ref{item:roman_II}:}}\quad 
		The proof proceeds along the same lines as that of Case \ref{item:roman_I}.  We construct an observable  $\Psi^* \in \mathbb{R}^2_{\widetilde{\Delta}}(\mathbb{T} )$ defined by
		\begin{equation}\label{PSI*}
			{\Psi ^*}\left( \theta  \right) = {\left( {\operatorname{Re} \sum\limits_{j = 1}^\infty  {\frac{{\exp \left( {2\pi i{q_{{m_j}}}\theta } \right)}}{{\widetilde\Delta \left( {{q_{{m_j}}}} \right)}}} ,\;\operatorname{Im} \sum\limits_{j = 1}^\infty  {\frac{{\exp \left( {2\pi i{q_{{m_j}}}\theta } \right)}}{{\widetilde\Delta \left( {{q_{{m_j}}}} \right)}}} } \right)^ \top }
		\end{equation}
		and employ a similar decomposition as described in \eqref{hebing}. Next, we consider a uniform estimate for \textit{every} $ \theta \in \mathbb T $. While the estimates for $\mathcal{S}_1(q_{m_s},\theta)$ and $\mathcal{S}_3(q_{m_s},\theta)$ remain fundamentally unchanged, the \textit{crucial} step lies in the treatment of $\mathcal{S}_2(q_{m_s},\theta)$, where, rather than $1$, we employ the auxiliary variable ${\exp(2\pi iq_{m_s}\theta)}$. In view of the approaches presented in \eqref{SSSS} and \eqref{CCC}, it follows that 
		\begin{align*}
			&\;\left| {\frac{1}{{{A_{{q_{{m_s}}}}}}}\sum\limits_{n = 0}^{{q_{{m_s}}} - 1} {w\left( {n/{q_{{m_s}}}} \right)\exp \left( {2\pi i{q_{{m_s}}}(\theta  + n\rho )} \right)}  - \exp \left( {2\pi i{q_{{m_s}}}\theta } \right)} \right|\notag \\
			=&\; \left| {\frac{1}{{{A_{{q_{{m_s}}}}}}}\sum\limits_{n = 0}^{{q_{{m_s}}} - 1} {w\left( {n/{q_{{m_s}}}} \right)\exp \left( {2\pi i{q_{{m_s}}}(\theta  + n\rho )} \right)}  - \frac{1}{{{A_{{q_{{m_s}}}}}}}\sum\limits_{n = 0}^{{q_{{m_s}}} - 1} {w\left( {n/{q_{{m_s}}}} \right)\exp \left( {2\pi i{q_{{m_s}}}\theta } \right)} } \right|\notag \\
			\leqslant &\;\frac{1}{{{A_{{q_{{m_s}}}}}}}\sum\limits_{n = 0}^{{q_{{m_s}}} - 1} {w\left( {n/{q_{{m_s}}}} \right)\left| {\exp \left( {2\pi i{q_{{m_s}}}\theta } \right)} \right|\left| {\exp \left( {2\pi i{q_{{m_s}}}n\rho } \right) - 1} \right|}  \notag \\
			\leqslant &\; \left( {1 - \frac{1}{{\sqrt 2 }}} \right) = \left( {1 - \frac{1}{{\sqrt 2 }}} \right)\left| {\exp \left( {2\pi i{q_{{m_s}}}\theta } \right)} \right|,\quad s \gg 1\notag .
		\end{align*}
		This leads to 
		\[\left| {\frac{1}{{{A_{{q_{{m_s}}}}}}}\sum\limits_{n = 0}^{{q_{{m_s}}} - 1} {w\left( {n/{q_{{m_s}}}} \right)\exp \left( {2\pi i{q_{{m_s}}}(\theta  + n\rho )} \right)} } \right| \geqslant \frac{1}{{\sqrt 2 }}\left| {\exp \left( {2\pi i{q_{{m_s}}}\theta } \right)} \right| = \frac{1}{{\sqrt 2 }}.\]
		For the real and imaginary parts of a complex number, we have
		\[
		\max\{|\operatorname{Re}z|,|\operatorname{Im}z|\}\geqslant \frac{|z|}{\sqrt2}.
		\]
		Hence we obtain
		\[
		\left\|\frac1{A_{q_{m_s}}}\sum_{n=0}^{q_{m_s}-1}w(n/q_{m_s})
		\begin{pmatrix}
			\cos(2\pi q_{m_s}(\theta+n\rho))\\
			\sin(2\pi q_{m_s}(\theta+n\rho))
		\end{pmatrix}\right\|_{\ell^\infty}\geqslant\frac12.
		\]
		Consequently, we still arrive at the estimate as in \eqref{S2GUJI}:
		\begin{equation}\label{S2GW} 
			\left\| {{\mathcal{S}_2}\left( {{q_{{m_s}}},\theta} \right)} \right\|_{\ell^\infty}  \geqslant \frac{1}{{2\widetilde \Delta \left( {{q_{{m_s}}}} \right)}}, \quad s \gg 1.
		\end{equation}
		Note that in this context, $s$ does not need to depend on $\theta$. Recalling that the estimates for $\mathcal{S}_1(q_{m_s}, \theta)$ and $\mathcal{S}_3(q_{m_s}, \theta)$ are also independent of $\theta$, we thereby conclude the proof for Case \ref{item:roman_II}.
		\vspace{3mm}

		\noindent \textbf{\textit{Proof of Case  \ref{item:roman_III}:}}\quad  Case \ref{item:roman_III} is a direct consequence of Case \ref{item:roman_II}. 
		For a fixed $\theta$, Case \ref{item:roman_II} gives a lower bound $1/(4\widetilde\Delta(q_{m_s}))$ for the total vector error at all sufficiently large selected times. By the infinite pigeonhole principle, one fixed component attains this bound at infinitely many such times. Choose that component of \eqref{PSI*} as $\Psi_*$. The component and the subsequence may depend on $\theta$, which proves Case \ref{item:roman_III}.
		\vspace{3mm}
		
		Finally, let us address the \textit{continuous} case in Theorem \ref{OPTT1}, which differs slightly from the previous discrete case in several aspects, such as the dimensionality, the construction of the observables, and the selection of the set. We consider Case \ref{item:roman_I} as an illustration; Cases \ref{item:roman_II} and \ref{item:roman_III} can be treated analogously.  For almost every $\rho=(\rho_1,\rho_2)\in\mathbb{R}^2$, signed coordinate permutations allow us to assume $0<\rho_2<\rho_1$ and $\rho_2/\rho_1\in(0,1)\setminus\mathbb{Q}$. Such torus automorphisms preserve the frequency supremum norm. By Fubini's theorem under the change of variables $(\rho_1,\rho_2/\rho_1)$, the ratio satisfies the required Borel--Bernstein property for almost every velocity.  In this context, the approximants $ {\left\{ {{p_n}/{q_n}} \right\}_{n \in {\mathbb{N}^ + }}} $ from the discrete setting of Case \ref{item:roman_I} are taken with respect to $\rho_1^{-1}\rho_2$,  while assumptions \ref{item:a}, \ref{item:b},
		and \ref{item:c}, together with the geometric growth condition
		on $\widetilde\Delta(q_{m_s})$, are imposed as before.
		In condition \ref{item:c}, take $C_1$ small enough that
		$C_1<e^{-1}\beta\lambda^{-1/\beta}
		\min\{1,\rho_1^{1/\beta}\}$.  It is then straightforward to verify that $p_n < q_n$ holds for all sufficiently large $n \in \mathbb{N}^+$. Next, we construct the required observable as 
		\[\Phi \left( \theta  \right) = \operatorname{Re} \sum\limits_{j = 1}^\infty  {\frac{{\exp \left( {2\pi i\left\langle {{k_{{m_j}}},\theta } \right\rangle } \right)}}{{\widetilde \Delta \left( {{{\left\| {{k_{{m_j}}}} \right\|}_{{\ell ^\infty }}}} \right)}}} ,\quad \theta  \in {\mathbb{T}^2}\]
		with $ {k_{{m_j}}} = \left( { - {p_{{m_j}}},{q_{{m_j}}}} \right) $ for all $ j \in \mathbb{N}^+ $.
		Clearly, since
		\[{\left\| {{k_{{m_j}}}} \right\|_{{\ell ^\infty }}} = \sup \left\{ {\left| {{p_{{m_j}}}} \right|,\left| {{q_{{m_j}}}} \right|} \right\} = {q_{{m_j}}},\quad \forall j \in {\mathbb{N}^ + },\]
		it follows that
		\begin{equation}\label{I-1}
			\widetilde \Delta \left( {{{\left\| {{k_{{m_j}}}} \right\|}_{{\ell ^\infty }}}} \right) = \widetilde \Delta \left( {{q_{{m_j}}}} \right),\quad \forall j \in {\mathbb{N}^ + }.
		\end{equation}
		Moreover, we have $ \int_{{\mathbb{T}^2}} {\Phi \left( x \right)dx}  = 0 $. Construct the set 
		\[\Theta : = \left\{ {\theta  \in {\mathbb{T}^2}:\quad \left( {{\theta _1},{\theta _2}} \right) = \left( {\left\{ {\iota {\rho _1}} \right\},\left\{ {\iota {\rho _2}} \right\}} \right),\;\iota  \in \mathbb{N}} \right\}.\]
		Clearly, $(0,0)\in\Theta$. To establish density, we may further
		exclude the set
		\[
		\bigcup_{k\in\mathbb Z^2\setminus\{0\}}\ \bigcup_{m\in\mathbb Z}
		\{\rho\in\mathbb R^2:\langle k,\rho\rangle=m\}.
		\]
		This is a countable union of affine lines and hence has Lebesgue
		measure zero. On its complement, $\langle k,\rho\rangle\notin\mathbb Z$
		for every nonzero $k\in\mathbb Z^2$, so Kronecker's
		theorem implies that the forward integer-time orbit $\Theta$ is
		dense in $\mathbb T^2$. This full-measure condition is compatible
		with the Borel--Bernstein property imposed above and is preserved
		by signed coordinate permutations. Notice that density of this
		integer-time orbit uses nonresonance modulo $\mathbb Z$, which is
		stronger than the nonvanishing condition for a continuous flow. For an initial point $ \theta = {\left( {\left\{ {\iota {\rho _1}} \right\},\left\{ {\iota {\rho _2}} \right\}} \right)} \in \Theta $ with some $ \iota \in \mathbb{N} $, let us consider the continuous version of the weighted Birkhoff average 
		\begin{equation}\label{I-2}
			\frac{1}{{{q_{{m_s}}}}}\int_0^{{q_{{m_s}}}} {w\left( {t/{q_{{m_s}}}} \right)\Phi \left( {{\mathscr{T}}_\rho ^t\left( \theta  \right)} \right)dt}  - \int_{{\mathbb{T}^2}} {\Phi \left( x \right)dx}
		\end{equation}
		and, following the approach in \eqref{hebing}, decompose it into three components $\mathcal{S}_1(q_{m_s}, \theta)$, $\mathcal{S}_2(q_{m_s}, \theta)$ and $\mathcal{S}_3(q_{m_s}, \theta)$. For brevity, their explicit expressions are omitted here. Regarding $\mathcal{S}_1(q_{m_s}, \theta)$, the analysis here is more straightforward than in the discrete case (as noted in \cite{TL24a}), as it obviates the need for the Poisson summation formula while yielding an identical estimate. For $\mathcal{S}_3(q_{m_s}, \theta)$, the treatment is entirely the same. The \textit{principal} difference lies in the estimate of $\mathcal{S}_2(q_{m_s}, \theta)$, since the estimate in \eqref{zhishu2} fails in the continuous setting. From 
		\[
		\left|q_{m_s}\frac{\rho_2}{\rho_1}-p_{m_s}\right|
		=\left\|q_{m_s}\frac{\rho_2}{\rho_1}\right\|_{\mathbb{Z}}
		<q_{m_s+1}^{-1},
		\]
		we immediately obtain
		\[
		|\langle k_{m_s},(\iota+t)\rho\rangle|
		=(\iota+t)|\rho_1|\left|q_{m_s}\frac{\rho_2}{\rho_1}-p_{m_s}\right|
		<(\iota+t)|\rho_1|q_{m_s+1}^{-1},\quad t\geqslant0,
		\]
		which, combined with \eqref{zhishu} and \eqref{CCC}, implies
		\begin{align*}
			&\;\left| {\frac{1}{{{q_{{m_s}}}}}\int_0^{{q_{{m_s}}}} {w\left( {t/{q_{{m_s}}}} \right)\exp \left( {2\pi i\left\langle {{k_{{m_s}}},\theta  + t\rho } \right\rangle } \right)dt}  - 1} \right|\\
			\leqslant&\; \frac{1}{{{q_{{m_s}}}}}\int_0^{{q_{{m_s}}}} {w\left( {t/{q_{{m_s}}}} \right)\left| {\exp \left( {2\pi i\left\langle {{k_{{m_s}}},\left( {\iota  + t} \right)\rho } \right\rangle } \right) - 1} \right|dt} \\
			\leqslant&\; \frac{{\widetilde C}}{{{q_{{m_s}}}}}\int_0^{{q_{{m_s}}}} {w\left( {t/{q_{{m_s}}}} \right)\left| {\left\langle {{k_{{m_s}}},\left( {\iota  + t} \right)\rho } \right\rangle } \right|dt} \\
			\leqslant&\; \frac{{\widetilde C|\rho_1|}}{{{q_{{m_s}}}}}\int_0^{{q_{{m_s}}}} {w\left( {t/{q_{{m_s}}}} \right)\left( {\iota  + t} \right)q_{{m_s} + 1}^{ - 1}dt} \\
			\leqslant &\;\frac{{2\widetilde C|\rho_1|}}{{{q_{{m_s}}}}}\int_0^{{q_{{m_s}}}} {w\left( {t/{q_{{m_s}}}} \right){q_{{m_s}}}q_{{m_s} + 1}^{ - 1}dt} \\
			\leqslant &\;\frac{1}{2}, \quad s \gg 1.
		\end{align*}
		The last inequality follows from $2\widetilde C|\rho_1|/m_s\to0$.
		Consequently, we still arrive at the estimate 
		\[{\mathcal{S}_2}\left( {{q_{{m_s}}},\theta } \right) \geqslant \frac{1}{{2\widetilde \Delta \left( {{q_{{m_s}}}} \right)}},\quad s \gg 1.\]
		The remainder of the argument proceeds exactly as in the discrete case. Combined with \eqref{I-1} and \eqref{I-2}, selecting $T=q_{m_s}$ provides a sequence that concludes the proof of the continuous case.
		\vspace{3mm}

		We now complete the proof of Theorem \ref{OPTT1}.

	\end{proof}

	\subsection{Proof of Corollary \ref{OPTCORO1}: Higher-dimensional case}\label{SECOPTCORO1}

	\begin{proof}
		Let $e_1=(1,0,\ldots,0)^\top\in\mathbb R^l$.
		By Fubini's theorem, for almost every $\rho\in\mathbb R^d$, the first coordinate $\rho_1$ satisfies the discrete assertion in Case \ref{item:roman_I} of Theorem \ref{OPTT1}.
		Let $\Psi$ and $\Xi$ be the corresponding scalar observable and dense set, respectively, and define
		\[
		\bar\Psi(\theta)=\Psi(\theta_1)e_1,
		\quad \theta\in\mathbb T^d,
		\qquad \bar\Xi=\Xi.
		\]
		The Fourier coefficients of $\bar\Psi$ satisfy
		\[
		\widehat{\bar\Psi}(k_1,0,\ldots,0)=\widehat\Psi(k_1)e_1,
		\]
		and vanish whenever $(k_2,\ldots,k_d)\neq(0,\ldots,0)$.
		Since $\|(k_1,0,\ldots,0)\|_{\ell^\infty}=|k_1|$, absolute summability is preserved, and the coefficient bound in Definition \ref{DEFBANACH} holds; hence $\bar\Psi\in\mathbb R^l_{\widetilde\Delta}(\mathbb T^d)$.
		Because both the spatial mean of $\bar\Psi$ and its evaluations along any orbit depend solely on the first coordinate, it follows that for every initial point,
		\[
		\mathbf{Error}_N^{\mathrm{dis}}(w,\bar\Psi,\rho,\theta)
		=\mathbf{Error}_N^{\mathrm{dis}}(w,\Psi,\rho_1,\theta_1).
		\]
		Consequently, the discrete lower bound holds at every point of $\bar\Xi\times\mathbb T^{d-1}$, which is a dense set containing the origin.
		
		Similarly, for almost every $\rho\in\mathbb R^d$ with $d\geqslant3$, the pair $(\rho_1,\rho_2)$ satisfies the continuous assertion in Case \ref{item:roman_I} of Theorem \ref{OPTT1}.
		Let $\Phi$ and $\Theta\subset\mathbb T^2$ denote the associated scalar observable and dense set, respectively, and set
		\[
		\bar\Phi(\theta)=\Phi(\theta_1,\theta_2)e_1,
		\quad \theta\in\mathbb T^d,
		\qquad \bar\Theta=\Theta.
		\]
		The Fourier coefficients of $\bar\Phi$ can be nonzero only when $k_3=\cdots=k_d=0$.
		Because appending these zero coordinates preserves the frequency supremum norm, we have $\bar\Phi\in\mathbb R^l_{\widetilde\Delta}(\mathbb T^d)$ and
		\[
		\mathbf{Error}_T^{\mathrm{con}}(w,\bar\Phi,\rho,\theta)
		=\mathbf{Error}_T^{\mathrm{con}}
		(w,\Phi,(\rho_1,\rho_2),(\theta_1,\theta_2)).
		\]
		This establishes the continuous lower bound on the dense cylinder $\bar\Theta\times\mathbb T^{d-2}$, which likewise contains the origin.
		For each fixed point in either base set, these identities yield the same sequence of times for the lower bound across all choices of the remaining coordinates.
		
		The proof of Corollary \ref{OPTCORO1} is complete.
	\end{proof}

	\subsection{Proof of Theorem \ref{OPTT2}: Optimality for universal uniform  convergence rates}\label{PFTHI}

	\begin{proof}		
		
		It suffices to consider the case $l=1$; the higher-dimensional case follows by a completely analogous argument. Consequently, the $\ell^\infty$-norm reduces to the standard absolute value.
		
		Note that the counterexample-type statements (i.e., the lower bound estimates for the convergence rates) in the stated parameter ranges follow directly from Theorem \ref{OPTT1} and Corollary \ref{OPTCORO1}; thus, it suffices to analyze the \textit{upper bound} estimates. A further clarification is required regarding the discrepancy in parameter requirements between the upper and lower bounds in Cases \ref{item3:roman_I} and \ref{item3:roman_IV}: the upper-bound assumptions are more restrictive in the former case, whereas the lower-bound assumptions are more restrictive in the latter, as mentioned in Remark \ref{RE37}. This is because the upper bound estimates necessitate the absolute convergence of Fourier coefficients for general observables across the \textit{entire} lattice $k \in \mathbb{Z}^d$. In contrast, for the lower bounds, we only examine specific counterexamples constructed in Theorem \ref{OPTT1} and Corollary \ref{OPTCORO1}. For these constructions, absolute convergence is only required over a \textit{subsequence} of $\mathbb{Z}^d$, which significantly relaxes the constraints. In particular, given the exponential properties of this  subsequence, our requirements in Theorem \ref{OPTT1} are entirely justified.\vspace{3mm}

		\noindent\textbf{{\textit{Proof of Case \ref{item3:roman_I}--\ref{TH3.3-I-1}:}}}\quad
		Assume $\widetilde\Delta(x)\sim x^\ell$ with $\ell>d\geqslant1$, and fix $1<\ell'<d^{-1}\ell$. We prove the upper bound simultaneously in all dimensions. The argument requires only $w\in C_0^M([0,1])$ with an integer $M\geqslant\max\{\ell,2\}$.
		
		By Remark \ref{fenli}, almost every $\rho\in\mathbb R^d$ satisfies \eqref{fgz} with $\Delta(x)=x^d\log^2(1+x)$. Fix such a vector $\rho$. Since $\ell>d$, the Fourier coefficients are absolutely summable, and hence
		\[\frac{1}{{{A_N}}}\sum\limits_{n = 0}^{N - 1} {w\left( {n/N} \right)f\left( {\mathscr{T}_\rho ^n\left( \theta  \right)} \right)}  - \int_{{\mathbb{T}^d}} {f\left( x \right) {d}x}  = \frac{1}{{{A_N}}}\sum\limits_{  k \in {\mathbb{Z}^d}\setminus\{0\}} {{{\widehat f}(k)}\sum\limits_{n = 0}^{N - 1} {w\left( {n/N} \right){\exp\left({2\pi i\left\langle {k,\theta  + n\rho } \right\rangle }\right)}} } \]
		for every $\theta\in\mathbb T^d$. Applying the Poisson summation formula as in the estimate of $\mathcal S_1(q_{m_s},\theta)$ in Section \ref{PFT1}, and using $N/A_N\to1$, we obtain
		\begin{align}
			&\mathbf{Error}_N^{\mathrm{dis}}(w,f,\rho,\theta)\notag\\*
			\leqslant{}&\frac1{A_N}\sum_{k\in\mathbb Z^d\setminus\{0\}}
			|\widehat f(k)|\left|\sum_{n=0}^{N-1}w(n/N)
			\exp(2\pi i\langle k,\rho\rangle n)\right|\notag\\
			\lesssim{}&\frac N{A_N}\sum_{k\in\mathbb Z^d\setminus\{0\}}
			\frac1{\|k\|_{\ell^\infty}^{\ell}}
			\left|\sum_{n\in\mathbb Z}\int_0^1w(z)
			\exp\bigl(2\pi Ni(\langle k,\rho\rangle-n)z\bigr)dz\right|\notag\\
			\lesssim{}&\sum_{k\in\mathbb Z^d\setminus\{0\}}
			\frac1{\|k\|_{\ell^\infty}^{\ell}}\sum_{n\in\mathbb Z}
			\min\left\{1,
			\frac{\|D^{\lceil\ell'\rceil}w\|_{L^1(0,1)}}
			{(2\pi N|\langle k,\rho\rangle-n|)^{\lceil\ell'\rceil}}
			\right\}\label{LISANFB4}\\
			\lesssim{}&\frac1{N^{\ell'}}\sum_{k\in\mathbb Z^d\setminus\{0\}}
			\frac1{\|k\|_{\ell^\infty}^{\ell}}\sum_{n\in\mathbb Z}
			\frac1{|\langle k,\rho\rangle-n|^{\ell'}}\label{LISANFB3}\\
			\lesssim{}&\frac1{N^{\ell'}}\sum_{k\in\mathbb Z^d\setminus\{0\}}
			\frac1{\|k\|_{\ell^\infty}^{\ell}}
			\left(\frac1{\|\langle k,\rho\rangle\|_{\mathbb Z}^{\ell'}}
			+2\sum_{n=1}^{\infty}\frac1{n^{\ell'}}+2^{\ell'}\right).
			\label{LISANFB2}
		\end{align}
		In \eqref{LISANFB4}, we combine the bound $\int_0^1w(z)dz=1$ with $\lceil\ell'\rceil$ integrations by parts; the endpoint conditions on $w$ eliminate the boundary terms, and $2\leqslant\lceil\ell'\rceil\leqslant M$. The inequality $\min\{1,C|u|^{-\lceil\ell'\rceil}\}\lesssim|u|^{-\ell'}$ for $u\ne0$ then gives \eqref{LISANFB3}. Separating the nearest integer from the remaining terms yields \eqref{LISANFB2}, where the series over $n$ converges because $\ell'>1$. These bounds also justify absolute convergence of the Poisson series for each fixed $k\ne0$.
		
		It remains to control the weighted small-divisor series in \eqref{LISANFB2}. For $R\geqslant1$, consider all integer frequencies in the box $\|k\|_{\ell^\infty}\leqslant R$, including $k=0$. Any two distinct frequencies $k,k'$ in this box satisfy $0<\|k-k'\|_{\ell^\infty}\leqslant2R$, so \eqref{fgz} gives
		\begin{equation}\label{eq:discrete-box-separation}
			\|\langle k,\rho\rangle-\langle k',\rho\rangle\|_{\mathbb Z}
			\geqslant\frac{\gamma}{\Delta(2R)}=:\delta_R.
		\end{equation}
		Represent the residues $\langle k,\rho\rangle\pmod{\mathbb Z}$ in $(-1/2,1/2]$. On each side of zero, order the nonzero representatives by increasing distance from zero. Pairwise separation, including separation from the residue at $k=0$, implies that the $j$-th distance is at least $j\delta_R$. Therefore, since $\ell'>1$, we have
		\begin{equation}\label{eq:discrete-box-sum}
			\sum_{0<\|k\|_{\ell^\infty}\leqslant R}
			\frac1{\|\langle k,\rho\rangle\|_{\mathbb Z}^{\ell'}}
			\leqslant2\delta_R^{-\ell'}\sum_{j=1}^{\infty}\frac1{j^{\ell'}}
			\lesssim R^{d\ell'}\log^{2\ell'}(1+R).
		\end{equation}
		We now apply \eqref{eq:discrete-box-sum} on the dyadic shells $2^j\leqslant\|k\|_{\ell^\infty}<2^{j+1}$, $j\in\mathbb N$, to obtain
		\begin{align}\label{LISANFB1}
			\sum_{k\in\mathbb Z^d\setminus\{0\}}
			\frac1{\|k\|_{\ell^\infty}^{\ell}
				\|\langle k,\rho\rangle\|_{\mathbb Z}^{\ell'}}
			&\leqslant\sum_{j=0}^{\infty}2^{-j\ell}
			\sum_{0<\|k\|_{\ell^\infty}\leqslant2^{j+1}}
			\frac1{\|\langle k,\rho\rangle\|_{\mathbb Z}^{\ell'}}\notag\\
			&\lesssim\sum_{j=0}^{\infty}2^{-j(\ell-d\ell')}(1+j)^{2\ell'}<+\infty,
		\end{align}
		because $\ell-d\ell'>0$. Together with $\ell>d$, which controls the remaining terms in \eqref{LISANFB2}, this proves the uniform $\mathcal O(N^{-\ell'})$ upper bound for every $d\geqslant1$ and every $1<\ell'<d^{-1}\ell$. \vspace{2mm}
		
		\noindent{\textit{An alternative proof for $d=1$:}}\quad The same upper bound can also be obtained from continued fractions and the Denjoy--Koksma inequality. Fix $\ell>1$ and $1<\ell'<\ell$. By \eqref{LISANFB2}, it suffices to prove the following estimate for almost every $\rho\in\mathbb R$, since the negative frequencies give the same contribution:
		\begin{equation}\label{LSZY1}
			\sum\limits_{k = 1}^\infty  {\frac{1}{{{k^\ell }\left\| {\left\langle {k,\rho } \right\rangle } \right\|_\mathbb{Z}^{\ell' }}}}  <  + \infty .
		\end{equation}
		To this end, let us first establish a key estimate.  Let $ \{p_n/q_n\}_{n \in \mathbb{N}^+} $ be the approximants  of $ \rho  \in {\mathbb{T} }\setminus \mathbb{Q}  $ (it suffices to consider this case) and define a function $ F \colon \mathbb{T} \to \mathbb{R}$ as 
		\[F(x):=\left\{ \begin{aligned}
			&{\left( {2{q_n}} \right)^{\ell' }},\quad &{\left\| x \right\|_\mathbb{Z}} \leqslant {\left( {2{q_n}} \right)^{ - 1}},  \hfill \\
			&\left\| x \right\|_\mathbb{Z}^{ - \ell' },\quad &{\left\| x \right\|_\mathbb{Z}} > {\left( {2{q_n}} \right)^{ - 1}}. \hfill \\
		\end{aligned}  \right.\]
		In view of \eqref{liangce} and \eqref{liangce22}, we have $ \left\| {\left\langle {k,\rho } \right\rangle } \right\|_{\mathbb{Z}} > {\left( {2{q_n}} \right)^{ - 1}} $ for all $ 0 < k < {q_n} $. Then, using the  Denjoy--Koksma inequality in Lemma \ref{DENJOYKOKSMA}, we obtain 
		\begin{align}\label{LSZY22}
			\sum\limits_{ k =1}^{q_n -1} {\frac{1}{{\left\| {\left\langle {k,\rho } \right\rangle } \right\|_\mathbb{Z}^{\ell' }}}} & = \sum\limits_{k = 1}^{{q_n} - 1} {F\left( {\left\langle {k,\rho } \right\rangle } \right)}  \leqslant \left| {F\left( 0 \right)} \right| + {q_n}\left| {\int_{{\mathbb{T} }} {F\left( x \right) {d}x} } \right| + \operatorname{Var}\left( F \right) \notag \\
			&= \mathcal{O}( {q_n^{\ell' }} )+\mathcal{O}( {q_n^{\ell' }} )+\mathcal{O}( {q_n^{\ell' }} ) \lesssim {q_n^{\ell '}} ,  \quad  n \gg 1.
		\end{align}
		This conclusion is analogous to    \cite[Lemma 3.2]{deF22} (see also    \cite[Lemma 2.5]{ALMN95} and   \cite[Lemma 1]{LM94}, and we also mention \cite{Rus76} for the continuous  case). Returning to the series in \eqref{LSZY1}, Abel's summation formula gives
		\begin{equation}\label{ZYLS3}
			\sum\limits_{k = 1}^{{q_L} - 1} {\frac{1}{{{k^\ell }\left\| {\left\langle {k,\rho } \right\rangle } \right\|_\mathbb{Z}^{\ell '}}}}  = \frac1{q_L^\ell}\sum\limits_{k = 1}^{{q_L} - 1} {\frac{1}{{\left\| {\left\langle {k,\rho } \right\rangle } \right\|_\mathbb{Z}^{\ell' }}}}  + \sum\limits_{j = 1}^{{q_L} - 1} {\left( {\sum\limits_{k = 1}^j {\frac{1}{{\left\| {\left\langle {k,\rho } \right\rangle } \right\|_\mathbb{Z}^{\ell' }}}} } \right)\left(\frac{1}{{{j^\ell }}}- {\frac{1}{{{{\left( {j + 1} \right)}^\ell }}} } \right)}
		\end{equation}
		for $ L \in \mathbb{N}^+$ sufficiently large. Then, with \eqref{LSZY22}, we have 
		\[\frac1{q_L^\ell}\sum\limits_{k = 1}^{{q_L} - 1} {\frac{1}{{\left\| {\left\langle {k,\rho } \right\rangle } \right\|_\mathbb{Z}^{\ell'} }}}  \lesssim \frac{q_L^{\ell\prime}}{q_L^\ell}  \ll 1, \quad  L\gg 1,\]
		and
		\begin{align*}
			\sum\limits_{j = 1}^{{q_L} - 1} {\left( {\sum\limits_{k = 1}^j {\frac{1}{{\left\| {\left\langle {k,\rho } \right\rangle } \right\|_\mathbb{Z}^{\ell'} }}} } \right)\left(\frac{1}{{{j^\ell }}}- {\frac{1}{{{{\left( {j + 1} \right)}^\ell }}} } \right)}&
			\lesssim   {\sum\limits_{v = 1}^L {\sum\limits_{j={q_{v - 1}}      }^{{q_v}-1} {\left( {\sum\limits_{k = 1}^j {\frac{1}{{\left\| {\left\langle {k,\rho } \right\rangle } \right\|_\mathbb{Z}^{\ell'} }}} } \right)\left(\frac{1}{{{j^\ell }}}- {\frac{1}{{{{\left( {j + 1} \right)}^\ell }}} } \right)} } } \\  
			&\lesssim    {\sum\limits_{v = 1}^L {\sum\limits_{j={q_{v - 1}}      }^{{q_v}-1} {q_v^{\ell'} \left(\frac{1}{{{j^\ell }}}- {\frac{1}{{{{\left( {j + 1} \right)}^\ell }}} } \right)} } }   \\
			&	\lesssim  {\sum\limits_{v = 1}^\infty  {\frac{{q_{v + 1}^{\ell'} }}{{q_v^\ell }}} }  , \quad  L\gg 1,
		\end{align*}
		hence the convergence in \eqref{LSZY1} will be ensured by \eqref{ZYLS3} if we prove $ \sum\nolimits_{v = 1}^\infty  {q_v^{ - \ell }q_{v + 1}^{\ell'} } <+\infty $. Recalling Remark \ref{fenli}, we still set the approximation function $ \Delta $ for almost every $ \rho \in \mathbb{R} $ as $ 	\Delta \left( x \right) = {x}{{\log }^2}\left( {1 + x} \right) $. Now, by \eqref{liangce} and \eqref{fgz}, we have
		\[\frac{1}{{{q_{v + 1}}}} > \left| {{q_v}\rho  - {p_v}} \right| = {\left\| {{q_v}\rho } \right\|_\mathbb{Z}} \geqslant \frac{\gamma }{{{q_v}\log^{2} \left( {1 + {q_v}} \right)}}.\]
		Then, with the exponential properties of $ \{q_n\}_{n \in \mathbb{N}^+} $, we obtain the desired boundedness as
		\[\sum\limits_{v = 1}^\infty  {\frac{{q_{v + 1}^{\ell '}}}{{q_v^\ell }}}  \lesssim {\sum\limits_{v = 1}^\infty  {\frac{{{{\left( {{q_v}\log^2(1+q_v)} \right)}^{\ell '}}}}{{q_v^\ell }}} }   \lesssim {\sum\limits_{v = 1}^\infty  {\frac{1}{{q_v^{\left( {\ell  - \ell '} \right)/2}}}} }  \lesssim {\sum\limits_{v = 1}^\infty  {\frac{1}{{{{({2^{\left( {\ell  - \ell '} \right)/4}})}^v}}}} }   <+\infty.\]
		This proves \eqref{LSZY1} and recovers the uniform $\mathcal O(N^{-\ell'})$ bound for every $1<\ell'<\ell$, completing the alternative one-dimensional proof. \vspace{2mm}

		On the other hand, it has been shown in Theorem \ref{OPTT1} and Corollary \ref{OPTCORO1} that in general, the uniform convergence rate cannot exceed $ \mathcal{O}( \widetilde \Delta (N)^{-1} )=\mathcal{O}( N^{-\ell} )  $ under the weighting function $w  \in \mathcal{W}_{\beta}$ for almost every  rotation and some specific observable. But, it remains to show such a result also holds for the weighting function $w \in C_0^M([0,1]) $, where $M\geqslant \max\{\ell,2\} $, see Condition \ref{item:num_1} of Definition \ref{generalwei}. To this end, let us revisit the proof of Theorem \ref{OPTT1}. Indeed, we only use the asymptotic behavior in Condition \ref{item:num_2} of Definition \ref{generalwei} for $w  \in \mathcal{W}_{\beta}$ in the analysis of $ 	{\mathcal{S}_1}\left( {{q_{{m_s}}},\theta} \right) $.  The  assumption $ \widetilde \Delta \left( x \right) \leqslant {\exp({x^\zeta })} $ with $ 0<\zeta <\beta^{-1} $, is also only used here, namely to ensure that assumption \ref{item:c} holds    to construct the desired increasing positive integer sequence  $ {\left\{ {{m_s}} \right\}_{s \in {\mathbb{N}^ + }}} $. However, when considering Case \ref{item3:roman_I}--\ref{TH3.3-I-1} (i.e., $ \widetilde\Delta(x)\sim x^\ell $  with  $ \ell>0 $) with  $  {w} \in  C_0^M([0,1]) $, we can set $ M_j=M$ there for all  $M\geqslant \max\{\ell,2\} $ ($ M=+\infty $ is  evident, we therefore only discuss the finite case), i.e., $ M_j $ are fixed for all $ 1 \leqslant j \leqslant s-1 $. 
		To include the endpoint $M=\ell$, write $\widehat w(u)=\int_0^1w(z)e^{2\pi iuz}dz$. Integration by parts and the Riemann--Lebesgue lemma give $\widehat w(u)=o(|u|^{-M})$. For each fixed $a\notin\mathbb{Z}$,
		\[
		\sum_{n\in\mathbb{Z}}|\widehat w(N(a-n))|=o(N^{-M}).
		\]
		Indeed, the terms multiplied by $N^M$ tend individually to zero and are dominated by the summable sequence $\|D^Mw\|_1(2\pi|a-n|)^{-M}$. Since $N/A_N\to1$, the same estimate holds for the normalized discrete multiplier. Retain assumptions \ref{item:a}, \ref{item:b}, and the geometric
		growth condition on \(\widetilde\Delta(q_{m_s})\), and replace
		\ref{item:c} by 
		\begin{enumerate}[label=(\alph*)]
			\setcounter{enumi}{3}
			\item\label{item:d}
			$\displaystyle
			\sum_{j=1}^{s-1}\frac1{\widetilde\Delta(q_{m_j})}
			\frac{q_{m_s}}{A_{q_{m_s}}}
			\left|\sum_{n\in\mathbb{Z}}\widehat w(q_{m_s}(q_{m_j}\rho-n))\right|
			\ll\frac1{\widetilde\Delta(q_{m_s})}.$
		\end{enumerate}
		At each inductive step there are only finitely many previously chosen frequencies. Their total contribution is $o(N^{-M})$, hence $o(N^{-\ell})$ when $M\geqslant\ell$. Thus the displayed condition can be enforced with the ratio to its right side at most $1/s$. In this way,
		one constructs the desired sequence $ \{m_{s}\}_{s\in \mathbb{N}^+} $, and then derives the estimate
		\[\left| {{\mathcal{S}_1}\left( {{q_{{m_s}}},\theta} \right)} \right|\leqslant \frac{1}{8{\widetilde \Delta \left( {{q_{{m_s}}}} \right)}}, \quad  s\gg 1, \]
		as done in \eqref{S19}. Consequently, we obtain the lower bound convergence rate $ \mathcal{O}(\widetilde{\Delta}(N)^{-1})=\mathcal{O}(N^{-\ell}) $ as in Theorem \ref{OPTT1} for $  {w} \in  C_0^M([0,1]) $. Now, following the same arguments in Section \ref{SECOPTCORO1}, we obtain a similar conclusion for Corollary \ref{OPTCORO1} as well, which provides the lower bound convergence rate for the general case on $ \mathbb{T}^d $. Moreover, it is evident that $M\geqslant\max\{\ell,2\}$ is sufficient to achieve the upper bound for the convergence rate, see \eqref{LISANFB4} for details. 
		
		This completes the proof of Case \ref{item3:roman_I}--\ref{TH3.3-I-1}. \vspace{3mm}
		
		\noindent\textbf{{\textit{Proof of Case \ref{item3:roman_I}--\ref{TH3.3-I-2}:}}} \quad  We next prove Case \ref{item3:roman_I}--\ref{TH3.3-I-2}, where $\widetilde{\Delta}(x) \sim x^\ell$ with $\ell > d$. Since the dimension significantly affects the results in this context, we discuss the cases separately below, building upon the framework of Theorem \ref{OPTT1} and the preceding Case \ref{item3:roman_I}--\ref{TH3.3-I-1}.\vspace*{2mm}
		
		\noindent{\textit{The  case $d=1$:}}\quad  It has been demonstrated in \cite[Theorem 3.4]{TL24a}  that the continuous case with $ d=1 $ exhibits  exponential convergence qualitatively,  owing to the absence of small divisors. This marks a fundamental distinction from the discrete case. Regarding the quantitative result here, by assuming $\rho\ne0$ and choosing an integer $M' \geqslant 2$ such that $M'\sim e^{-1}(2\pi|\rho|T/\lambda)^{1/\beta}$ for $T$ sufficiently large, we establish the desired conclusion:
		\begin{align}
			\mathbf{Error}_T^{\mathrm{con}}(w, f, \rho, \theta)   &\lesssim\sum\limits_{k \in \mathbb{Z} \setminus \left\{ 0 \right\}} {\frac{1}{{\widetilde\Delta \left( {{{\left| k \right|}}} \right)}}\frac{{{{ \| {{D^{M'}}w}  \|}_{{L^1}\left( {0,1} \right)}}}}{{{{\left| {2\pi k\rho T} \right|}^{M'}}}}}  \notag \\
			& \lesssim\mathop{\mathrm{inf}}\limits_{2 \leqslant M' \in {\mathbb{N}^ + }} \frac{{{\lambda ^{M'}}{{\left( {M'} \right)}^{\beta M'}}}}{{{{\left| {2\pi \rho T} \right|}^{M'}}}},\notag \\
			& =o\bigl(\exp(-T^{\varrho})\bigr),\quad \forall \theta \in \mathbb{T}\notag 
		\end{align}
		for any $ 0<\varrho <\beta^{-1} $.  The case $d=1$ is thus proved.
		\vspace{2mm}

		\noindent{\textit{The  case $d\geqslant 2$:}}\quad  As in Case \ref{item3:roman_I}--\ref{TH3.3-I-1} (see \eqref{LISANFB2}), we only need to show that for the continuous case with $ d \geqslant 2 $, under  the assumption that $\ell>d $ and $ 1 < \ell ' < {\left( {d - 1} \right)^{ - 1}}\ell  $, the following holds for almost every $ \rho \in \mathbb{R}^d $:
		\begin{equation}\label{LXZY1}
			\sum\limits_{  k \in {\mathbb{Z}^d}\setminus\{0\}} {\frac{1}{{{{{{{\left\| k \right\|}_{{\ell ^\infty }}^\ell}}} }\left| {\left\langle {k,\rho } \right\rangle } \right|^{\ell '}}}}  <  + \infty .
		\end{equation}
		Once this is established, we directly obtain the convergence rate $\mathcal{O}(T^{-\ell'})$ for the continuous version of the weighted Birkhoff average with $d \geqslant 2$.

		In what follows, we introduce a novel approach inspired by \cite{Sal04} to estimate \eqref{LXZY1}\footnote{\cite{Sal04} dealt with homological equations arising from KAM theory, namely $ \sum\nolimits_{  k \in {\mathbb{Z}^d}\setminus\{0\}} {{{\left| {\left\langle {k,\omega } \right\rangle } \right|}^{ - 1}}{\exp\left({ - \epsilon {{{\left\| k \right\|}_{{\ell ^\infty }}}} }\right)}}  $ for a given Diophantine vector $ \omega \in \mathbb{R}^d $ and some $ 0<\epsilon\ll 1 $, which is fundamentally different from our setting.}.   The key observation is that only a few of the divisors $ \left| {\left\langle {k,\rho } \right\rangle } \right| $ are actually small; i.e., near-resonances are relatively rare. As in \eqref{eq:discrete-box-sum}, the small divisors are estimated collectively. Here we use their absolute values $|\langle k,\rho\rangle|$ in place of distances modulo $\mathbb Z$. By Remark \ref{fenli}, we set the nonresonance condition \eqref{fgz2} for almost every $ \rho \in \mathbb{R}^d $ as 
		\begin{equation}\label{LXZY999}
			| {\left\langle {k,\rho } \right\rangle  } | \geqslant \frac{\gamma }{{{{{\left\| k \right\|}_{{\ell ^\infty }}^{d - 1}}}}{\log ^2}\left( {1 + {{{\left\| k \right\|}_{{\ell ^\infty }}}}} \right)}, \quad \forall   k \in {\mathbb{Z}^d}\setminus\{0\}.
		\end{equation}
		We may decrease $\gamma$ so that $0<\gamma\leqslant e^{-4}\log^2(2)/4$. We now proceed with the analysis by distinguishing between cases \ref{item:A2} and \ref{item:B2}.
		\begin{enumerate}[label=(\Alph*), wide=0pt, listparindent=\parindent] 
			\item    \label{item:A2}	 We first claim that terms in \eqref{LXZY1} with $ \left| {\left\langle {k,\rho } \right\rangle } \right| \geqslant ( {2{{\log }^2}2} )^{-1}\gamma  $ are well-controlled due to the absence of small divisors. Indeed, the following holds since $ \ell>d $:
			\[\sum\limits_{  k \in {\mathbb{Z}^d}\setminus\{0\},\;\left| {\left\langle {k,\rho } \right\rangle } \right| \geqslant \frac{\gamma }{{2{{\log }^2}2}}} {\frac{1}{{{{{{\left\| k \right\|}_{\ell ^\infty }^\ell}} }{{\left| {\left\langle {k,\rho } \right\rangle } \right|}^{\ell '}}}}}  \leqslant {\left( {\frac{{2{{\log }^2}2}}{\gamma }} \right)^{\ell '}}\sum\limits_{  k \in {\mathbb{Z}^d}\setminus\{0\}} {\frac{1}{{{{\left\| k \right\|}_{\ell ^\infty }^\ell }}}}  \lesssim {\sum\limits_{j = 1}^\infty  {\frac{1}{{j^{\ell-d+1}}}} }   <+\infty.\]
			
			\item    \label{item:B2} It remains to deal with terms with $ \left| {\left\langle {k,\rho } \right\rangle } \right| < ( {2{{\log }^2}2} )^{-1}\gamma$, where small divisors appear. For fixed $ \nu \in \mathbb{N}^+ $ and $ y>0 $, define the set
			\[\mathscr{S}\left( {\nu ,y} \right): = \left\{ {k \in {\mathbb{Z}^d}:\quad \frac{1}{{{2^{\nu  + 1}}}} \cdot \frac{{\gamma }}{{{{\log }^2}2}} \leqslant \left| {\left\langle {k,\rho } \right\rangle } \right| < \frac{1}{{{2^\nu }}} \cdot \frac{{\gamma }}{{{{\log }^2}2}},\;0 < \| k \|_{\ell^\infty} \leqslant y} \right\}.\]
			We first estimate the cardinality of $ \mathscr{S}\left( {\nu ,y} \right) $. Without loss of generality, assume $ \left| {{\rho _j}} \right| \leqslant \left| {{\rho _d}} \right| $ for all $ 1 \leqslant j \leqslant d $. Then, by taking $ k = \left( {0,  \ldots  ,0,1} \right) $, we obtain $ \left| {{\rho _d}} \right| \geqslant ({\log ^2}2)^{-1} \gamma    $ via  \eqref{LXZY999}. Denote by $ \bar k  = \left( {{k_1}, \ldots ,{k_{d - 1}}} \right) $ the projection of $k$ onto the first $d-1$ coordinates. We assert that $ \bar k  \ne \bar {k'}  $ whenever $ k,k' \in \mathscr{S}\left( {\nu ,y} \right) $ and $ k \ne k' $. Otherwise, we reach a contradiction:
			\begin{equation*} 
				\left\{ \begin{gathered}
					\left| {\left\langle {k - k',\rho } \right\rangle } \right| \leqslant \left| {\left\langle {k,\rho } \right\rangle } \right| + \left| {\left\langle {k',\rho } \right\rangle } \right| < \frac{2}{{{2^\nu }}} \cdot \frac{\gamma }{{{{\log }^2}2}} \leqslant \frac{\gamma }{{{{\log }^2}2}},  \hfill \\
					\left| {\left\langle {k - k',\rho } \right\rangle } \right| = \left| k_d-k_d^\prime \right|\left| {{\rho _d}} \right| \geqslant \left| {{\rho _d}} \right| \geqslant \frac{\gamma }{{{{\log }^2}2}} .  \hfill \\
				\end{gathered}  \right.
			\end{equation*}
			For a fixed $\bar k=(k_1,\ldots,k_{d-1})\ne0$, choose
			\[
			k_d^*\in\operatorname*{argmin}_{n\in\mathbb{Z}}
			\left|\sum_{j=1}^{d-1}k_j\rho_j+n\rho_d\right|,
			\quad k^*=(\bar k,k_d^*).
			\]
			Thus $k_d^*$ is a nearest integer to $-\rho_d^{-1}\sum_{j<d}k_j\rho_j$;
			hence 
			\[\left| {k_d^ * } \right| \leqslant 1 + \sum\limits_{j = 1}^{d - 1} {\left| {{k_j}} \right|\left| {{\rho _j}{\rho _d^{-1}}} \right|}  \leqslant 1 + \sum\limits_{j = 1}^{d - 1} {\left| {{k_j}} \right|}  \leqslant d \| {\bar k}  \|_{\ell^\infty},\]
			which leads to 
			\begin{equation}\label{LXZY33}
				\left\| {{k^ * }} \right\|_{\ell^\infty} \leqslant \sup \left\{ \| {\bar k}  \|_{\ell^\infty} , |k_d^*|\right\} \leqslant d \| {\bar k}  \|_{\ell^\infty}.
			\end{equation}
			From the definition of $ k^* $ and  \eqref{LXZY33}, it follows that for every $k$ with $ \bar k \ne 0 $,
			\begin{equation*} 
				\left| {\left\langle {k,\rho } \right\rangle } \right| \geqslant \left| {\left\langle {{k^*},\rho } \right\rangle } \right| \geqslant \frac{\gamma }{{\left\| {{k^*}} \right\|_{{\ell ^\infty }}^{d - 1}{{\log }^2}\left( {1 + {{\left\| {{k^*}} \right\|}_{{\ell ^\infty }}}} \right)}} \geqslant \frac{\gamma }{{{{\left( {d{{ \| {\bar k}  \|}_{{\ell ^\infty }}}} \right)}^{d - 1}}{{\log }^2}\left( {1 + d{{ \| {\bar k}  \|}_{{\ell ^\infty }}}} \right)}},
			\end{equation*}
			implying
			\begin{equation*} 
				\|\bar k\|_{\ell^\infty}\gtrsim
				\left(\frac1{|\langle k,\rho\rangle|
					[\log(e+|\langle k,\rho\rangle|^{-1})]^2}\right)^{1/(d-1)}.
			\end{equation*}
			This inverse estimate is applied only at the small-divisor scales fixed above. Distinct elements $k,k'\in\mathscr S(\nu,y)$ have distinct projections, and it follows that
			\[
			\|\bar k-\bar k'\|_{\ell^\infty}\gtrsim
			\left(\frac{2^\nu}{\nu^2}\right)^{1/(d-1)}.
			\]
			The same estimate applied to $k$ itself separates its projection from zero: $\bar k=0$ would contradict $|\langle k,\rho\rangle|<|\rho_d|$. Hence a nonempty $\mathscr S(\nu,y)$ has $y\gtrsim(2^\nu/\nu^2)^{1/(d-1)}$. The usual packing bound, initially with a $1+$ term, therefore gives the following bound without that term.
			Consequently, the cardinality of $\mathscr{S}\left( {\nu ,y} \right)$ can be estimated as
			\begin{equation}\label{LXZY34}
				{\mathrm{Card}} \mathscr{S}\left( {\nu ,y} \right) \lesssim{\nu ^2}{2^{ - \nu }}{y^{d - 1}}.
			\end{equation}
			Let $ \mathscr{S}\left( y \right) := \bigcup_{\nu\in \mathbb{N}^+}  {\mathscr{S}\left( {\nu ,y} \right)}  $. With these preparations, we estimate the remaining terms in \eqref{LXZY1} as
			\begin{align}
				\sum\limits_{  k \in {\mathbb{Z}^d}\setminus\{0\},\;\left| {\left\langle {k,\rho } \right\rangle } \right| < \frac{\gamma }{{2{{\log }^2}2}}} {\frac{1}{{{{\| k \|}^\ell_{\ell^\infty} }{{\left| {\left\langle {k,\rho } \right\rangle } \right|}^{\ell '}}}}}  &= \sum\limits_{j = 1}^\infty  {\sum\limits_{  k \in {\mathbb{Z}^d}\setminus\{0\},\;\left| {\left\langle {k,\rho } \right\rangle } \right| < \frac{\gamma }{{2{{\log }^2}2}},\;{\| k \|_{\ell^\infty}} = j} {\frac{1}{{{j^\ell }{{\left| {\left\langle {k,\rho } \right\rangle } \right|}^{\ell '}}}}} } \notag \\
				& = \sum\limits_{j = 1}^\infty  {\sum\limits_{k \in \mathscr{S}\left( j \right)} {\frac{1}{{{{\left| {\left\langle {k,\rho } \right\rangle } \right|}^{\ell '}}}}\left( {\frac{1}{{{j^\ell }}} - \frac{1}{{{{\left( {j + 1} \right)}^\ell }}}} \right)} } \notag \\
				& = \sum\limits_{j = 1}^\infty  {\left( {\frac{1}{{{j^\ell }}} - \frac{1}{{{{\left( {j + 1} \right)}^\ell }}}} \right)\sum\limits_{\nu  = 1}^\infty  {\sum\limits_{k \in \mathscr{S}\left( {\nu ,j} \right)} {\frac{1}{{{{\left| {\left\langle {k,\rho } \right\rangle } \right|}^{\ell '}}}}} } } \notag \\
				\label{LXZY16}& \leqslant \sum\limits_{j = 1}^\infty  {\left( {\frac{1}{{{j^\ell }}} - \frac{1}{{{{\left( {j + 1} \right)}^\ell }}}} \right)\sum\limits_{\nu  = 1}^\infty  {{{\left( {\frac{{2{{\log }^2}2}}{\gamma }} \right)}^{\ell '}}{2^{\nu \ell '}} \cdot 	{\mathrm{Card}} \mathscr{S}\left( {\nu ,j} \right)} } \\
				\label{LXZY17}& \lesssim\sum\limits_{j = 1}^\infty  {\left( {\frac{1}{{{j^\ell }}} - \frac{1}{{{{\left( {j + 1} \right)}^\ell }}}} \right){j^{d - 1}}\sum\limits_{{2^\nu }{{\log }^2}2 \leqslant {j^{d - 1}}{{\log }^2}\left( {1 + j} \right)} {{\nu ^2}{2^{\nu \left( {\ell ' - 1} \right)}}} } \\
				\label{LXZY18}& \lesssim\sum\limits_{j = 1}^\infty  {\left( {\frac{1}{{{j^\ell }}} - \frac{1}{{{{\left( {j + 1} \right)}^\ell }}}} \right){j^{d - 1}}\log^2(1+j){{\left( {{j^{d - 1}}\log^2(1+j)} \right)}^{ {\ell ' - 1}  }}} \\
				\label{LXZY19}& \lesssim\sum\limits_{j = 1}^\infty  {\frac{{{{\left( {\log(1+j)} \right)}^{2\ell '}}}}{{{j^{\ell  - \left( {d - 1} \right)\ell ' + 1}}}}} ,
			\end{align}
			where \eqref{LXZY16} follows from the definition of $ {\mathscr{S}\left( {\nu ,j} \right)} $, and \eqref{LXZY17} utilizes \eqref{LXZY34} and the fact that for $ {k \in \mathscr{S}\left( {\nu ,j} \right)} $, ${{2^\nu }{{\log }^2}2 \leqslant {j^{d - 1}}{{\log }^2}\left( {1 + j} \right)}$ since
			\[\frac{\gamma }{{{j^{d - 1}}{{\log }^2}\left( {1 + j} \right)}} \leqslant \frac{\gamma }{{{{\| k \|}^{d - 1}_{\ell^\infty}}{{\log }^2}\left( {1 + \| k \|_{\ell^\infty}} \right)}} \leqslant \left| {\left\langle {k,\rho } \right\rangle } \right| \leqslant \frac{1}{{{2^\nu }}} \cdot \frac{\gamma }{{{{\log }^2}2}}.\]
			Inequality \eqref{LXZY18} follows from the summation of the geometric-like series (note $\ell'>1$). Finally, \eqref{LXZY19} is bounded due to our choice of $\ell' < (d-1)^{-1}\ell$.
			
			Combining the estimates in \ref{item:A2} and \ref{item:B2}, we establish \eqref{LXZY1}, thereby obtaining the desired convergence rate $\mathcal{O}(T^{-\ell'})$ for the case $d\geqslant 2$.	
		\end{enumerate}	
		
		This completes the proof of Case \ref{item3:roman_I}--\ref{TH3.3-I-2}. 
		\vspace{3mm}

		\noindent\textbf{{\textit{Proof of Case \ref{item3:roman_I}--\ref{TH3.3-I-3}:}}}\quad  In this context, we have $ \widetilde \Delta \left( x \right) \sim {x^\ell }{\left( {\log x} \right)^\eta } $  with $ \ell  \geqslant  2 $ and $ \eta  > 2\ell  + 1 $,   yielding $ \sum\nolimits_{k \in \mathbb{Z}^2 \setminus \left\{ 0 \right\}} {\widetilde\Delta {{\left( {{{\left\| k \right\|}_{{\ell ^\infty }}}} \right)}^{ - 1}}}  <  + \infty  $. The proof follows from the analysis in Case \ref{item3:roman_I}--\ref{TH3.3-I-2} with $d=2$.   Under the above setting (especially for the nonresonance condition), it suffices to show that $\ell'$ can reach $\ell$; that is, the condition
		\begin{equation}\label{OPT41}
			\sum\limits_{  k \in {\mathbb{Z}^2}\setminus\{0\}} {\frac{1}{{{{\left\| k \right\|}^\ell_{\ell^\infty} }{{  {\log^\eta (1 + \left\| k \right\|_{\ell^\infty})}  } }{{\left| {\left\langle {k,\rho } \right\rangle } \right|}^\ell }}}}  <  + \infty
		\end{equation}
		guarantees a convergence rate of $ \mathcal{O}(T^{-\ell}) $. Similar to \ref{item:A2}, with $ \ell  \geqslant d=2 $ and $ \eta  > 2\ell  + 1 > 1 $, we obtain
		\begin{align*}
			\sum\limits_{  k \in {\mathbb{Z}^2}\setminus\{0\},\;\left| {\left\langle {k,\rho } \right\rangle } \right| \geqslant \frac{\gamma }{{2{{\log }^2}2}}} {\frac{1}{{{{\left\| k \right\|}^\ell_{\ell^\infty} }{{  {\log^\eta (1 + \left\| k \right\|_{\ell^\infty})}  } }{{\left| {\left\langle {k,\rho } \right\rangle } \right|}^\ell }}}}  &\lesssim {\sum\limits_{  k \in {\mathbb{Z}^2}\setminus\{0\}} {\frac{1}{{{{\left\| k \right\|}^\ell_{\ell^\infty} }{{  {\log^\eta (1 + \left\| k \right\|_{\ell^\infty})}  }  }}}} }  \\
			&   \lesssim  {\sum\limits_{j = 2}^\infty  {\frac{1}{{j^{\ell-1}{{\left( {\log j} \right)}^{\eta }}}}} }  <+\infty.
		\end{align*}
		Following the framework of \ref{item:B2} (see \eqref{LXZY18}), use the explicit comparison function
		\[
		H(x)=x^\ell(\log(1+x))^\eta,\quad x\geqslant1.
		\]
		Since $\widetilde\Delta(x)\sim H(x)$, it suffices to estimate the Fourier coefficients with the envelope $C/H(\|k\|_{\ell^\infty})$. For this model,
		\[
		\frac{H'(x)}{H(x)^2}\lesssim x^{-\ell-1}(\log(1+x))^{-\eta}.
		\]
		Consequently, summation by parts gives
		\begin{align*}
			\sum_{j=1}^{\infty}\left(\frac1{H(j)}-\frac1{H(j+1)}\right)\bigl(j\log^2(1+j)\bigr)^\ell
			&\lesssim\sum_{j=1}^{\infty}\frac{(j\log^2(1+j))^\ell}{j^{\ell+1}(\log(1+j))^\eta}\\
			&=\sum_{j=1}^{\infty}\frac1{j(\log(1+j))^{\eta-2\ell}}<+\infty.
		\end{align*}
		Here the differences of $H^{-1}$ are bounded by integrating $H'/H^2$ on $[j,j+1]$. This controls the remaining terms in \eqref{OPT41} without assuming a derivative bound for the original $\widetilde\Delta$.
		This completes the proof of Case \ref{item3:roman_I}--\ref{TH3.3-I-3}.
		\vspace{3mm}

		\noindent\textbf{{\textit{Proof of Case \ref{item3:roman_II}:}}}\quad 
		It suffices to observe that the integrability condition \eqref{keji} in Theorem \ref{OPTYINDL} holds for any given integer $m \geqslant 2$, provided that $\widetilde{\Delta}(x) \gg x^L$  for arbitrarily large $L > 0$. This implies uniform convergence rates of $\mathcal{O}(N^{-m})$ and $\mathcal{O}(T^{-m})$ for any arbitrary polynomial order. However, the control coefficient may diverge to $+\infty$ as $m \to +\infty$; for further details, refer to \cite[Section 1.1.4]{TL26a}. 
		This completes the proof of Case \ref{item3:roman_II}.
		\vspace{3mm}
		
		\noindent\textbf{{\textit{Proofs of Cases \ref{item3:roman_III} and \ref{item3:roman_IV}:}}}\quad Although the approaches for Cases \ref{item3:roman_III} and \ref{item3:roman_IV} differ from the preceding ones, their proofs are analogous and proceed along the lines of the argument for Theorem \ref{OPTT1}. We only present the discrete setting, as the continuous case is much simpler.
		
		We first prove Case \ref{item3:roman_III}.	In this case, we have $\widetilde{\Delta}(x) \sim \exp((\log x)^{\lambda})$ with $\lambda > 1$.  For almost every  $ \rho \in \mathbb{R}^d $, we set the approximation function $ \Delta $  as $ 	\Delta \left( x \right) = {x^{d }}{{\log }^2}\left( {1 + x} \right) $; specifically, the following nonresonance condition is satisfied (where we retain the notation $\Delta$ for brevity):
		\[{\left\| {\left\langle {k,\rho } \right\rangle } \right\|_\mathbb{Z}} > \frac{\gamma }{{\Delta \left( {\left\| k \right\|_{\ell^\infty}} \right)}}, \quad \forall   k \in {\mathbb{Z}^d}\setminus\{0\}.\]
		It then follows that
		\begin{equation}\label{J1J2}
			\mathbf{Error}_N^{\mathrm{dis}}\left(w,f,\rho,\theta\right)   \lesssim {\mathcal{J}_1}\left( {N,\theta } \right) + {\mathcal{J}_2}\left( {N,\theta } \right) ,\quad\forall \theta \in \mathbb{T}^d,
		\end{equation}
		where
		\begin{align*}
			{\mathcal{J}_1}\left( {N,\theta } \right) &:= \sum\limits_{0 < \left\| k \right\|_{\ell^\infty} \leqslant \varphi \left( N \right)} {\frac{1}{{\widetilde\Delta \left( {{{\left\| k \right\|}_{\ell^\infty} }} \right)}}\frac{1}{{{A_N}}}\left|\sum\limits_{n = 0}^{N - 1} {w\left( {n/N} \right){\exp\left({2\pi i\left\langle {k,\rho } \right\rangle n}\right)} \cdot {\exp\left({2\pi i\left\langle {k,\theta } \right\rangle }\right)}}\right| }, \\
			{\mathcal{J}_2}\left( {N,\theta } \right) &:= \sum\limits_{\left\| k \right\|_{\ell^\infty} > \varphi \left( N \right)} {\frac{1}{{\widetilde\Delta \left( {{{\left\| k \right\|}_{\ell^\infty} }} \right)}}\frac{1}{{{A_N}}}\left|\sum\limits_{n = 0}^{N - 1} {w\left( {n/N} \right){\exp\left({2\pi i\left\langle {k,\rho } \right\rangle n}\right)} \cdot {\exp\left({2\pi i\left\langle {k,\theta } \right\rangle }\right)}}\right| } ,
		\end{align*}
		and the function $ \varphi \left( N \right) $ for $ N\gg 1 $ is chosen as 
		\begin{equation}\label{fain2}
			\varphi(N):=N^{1/d}(\log N)^{-[\beta(1+\lambda)+2]/d}.
		\end{equation}
		For the weight, write $\lambda_w$ for the derivative-growth constant denoted by $\lambda$ in Definition \ref{generalwei}; it is independent of the logarithmic exponent $\lambda$ in the present case. With \eqref{fain2}, $\Delta(\varphi(N))\sim d^{-2}N(\log N)^{-\beta(1+\lambda)}$. Thus $(N/\Delta(\varphi(N)))^{1/\beta}\sim d^{2/\beta}(\log N)^{1+\lambda}$, and we derive
		\begin{equation}\label{gujie12}
			\exp \left( { - {C_2}{{\left( {\frac{N}{{\Delta \circ \varphi \left( N \right)}}} \right)}^{{\beta ^{ - 1}}}}} \right)\lesssim {\exp \left( - {C_3 \left(\log N\right)^{\lambda} }\right)} ,\quad \forall\,0<C_2<e^{-1}\beta(2\pi\gamma/\lambda_w)^{1/\beta},
		\end{equation}
		and
		\begin{equation}\label{gujie222}
			\varphi {\left( N \right)^d}{\left( {\log \varphi \left( N \right)} \right)^{ - \lambda  + 1}}\exp \left( { - {{\left( {\log \varphi \left( N \right)} \right)}^\lambda }} \right) \lesssim \exp \left( { - {C_4}{{\left( {\log N} \right)}^\lambda }} \right),
		\end{equation}
		for some constants $ C_3,C_4 >0$ (note that $ C_3 $ also depends on $ C_2 $).
		With these in mind, we proceed to \textit{quantitatively} estimate $\mathcal{J}_1(N, \theta)$ and $\mathcal{J}_2(N, \theta)$ in \eqref{J1J2}.

		Following the estimate for $ {\mathcal{S}_1}\left( {{q_{{m_s}}},\theta} \right) $ in Case \ref{item:roman_I} of Section \ref{PFT1}, we obtain the following bound for $\mathcal{J}_1(N, \theta)$ for all $\theta \in \mathbb{T}^d$:
		\begin{align}
			{{\mathcal{J}_1}\left( {N,\theta } \right)}& \lesssim\sum\limits_{0 < \left\| k \right\|_{\ell^\infty} \leqslant \varphi \left( N \right)} {\frac{1}{{\widetilde{\Delta} \left( {{{\left\| k \right\|}_{\ell^\infty} }} \right)}}\sum\limits_{n \in \mathbb{Z} }  {\frac{{{{\left\| {{D^{{{K_k}}}w}} \right\|}_{{L^1}\left( {0,1} \right)}}}}{{{{\left( {2\pi N\left| {\left\langle {k,\rho } \right\rangle  - n} \right|} \right)}^{{K_k}}}}}} } 	\label{MJ2}  \\
			& \lesssim\sum\limits_{0 < \left\| k \right\|_{\ell^\infty} \leqslant \varphi \left( N \right)} {\frac{1}{{\widetilde{\Delta} \left( {{{\left\| k \right\|}_{\ell^\infty} }} \right)}}\frac{{\lambda_w^{K_k}K_k^{{K_k}\beta }{\Delta }\left( {\left\| k \right\|_{\ell^\infty}} \right)^{{K_k}}}}{{{{\left( {2\pi \gamma N} \right)}^{{K_k}}}}}} \notag \\
			& \lesssim\sum\limits_{0 < \left\| k \right\|_{\ell^\infty} \leqslant \varphi \left( N \right)} {\frac{1}{{\widetilde{\Delta} \left( {{{\left\| k \right\|}_{\ell^\infty} }} \right)}}\exp \left( { - C_{2}{{\left( {\frac{{  N}}{{\Delta  \circ \varphi \left( N \right)}}} \right)}^{\beta^{-1}}}} \right)} \notag \\
			& \lesssim\exp \left( { - C_{2} {{\left( {\frac{{  N}}{{\Delta  \circ \varphi \left( N \right)}}} \right)}^{\beta^{-1}}}} \right)\notag \\
			\label{o11}&\lesssim \exp \left( { - {C_3}{{\left( {\log N} \right)}^\lambda }} \right) ,\quad N \gg 1,
		\end{align}
		where the integers $K_k \geqslant 2$ in \eqref{MJ2} are chosen for each nonzero $k$ with $\|k\|_{\ell^\infty}\leqslant\varphi(N)$ following \eqref{MJ}, with $K_k\sim e^{-1}(2\pi\gamma N/[\lambda_w\Delta(\|k\|_{\ell^\infty})])^{1/\beta}$, and \eqref{o11} follows from \eqref{gujie12}.
		
		On the other hand, it is straightforward to verify that for all $\theta \in \mathbb{T}^d$, the term $\mathcal{J}_2(N, \theta)$ satisfies the estimate:
		\begin{align}
			{{\mathcal{J}_2}\left( {N,\theta } \right)} &\leqslant \sum\limits_{\left\| k \right\|_{\ell^\infty}> \varphi \left( N \right)} {\frac{1}{{\widetilde{\Delta} \left( {{{\left\| k \right\|}_{\ell^\infty} }} \right)}}\left( {\frac{1}{{{A_N}}}\sum\limits_{n = 0}^{N - 1} {w\left( {n/N} \right)} } \right)} \notag \\
			&  \lesssim  \sum\limits_{\left\| k \right\|_{\ell^\infty} > \varphi \left( N \right)} {\frac{1}{{ \exp \left( {{{\left( {\log {{\left\| k \right\|}_{{\ell ^\infty }}}} \right)}^\lambda }} \right) }}} \notag \\
			&\lesssim\int_{\varphi \left( N \right)}^{ + \infty } {{r^{d - 1}}\exp \left( { - {{\left( {\log r} \right)}^\lambda }} \right)dr} \notag \\
			\label{gujie22}&  \lesssim \varphi {\left( N \right)^d}{\left( {\log \varphi \left( N \right)} \right)^{ - \lambda  + 1}}\exp \left( { - {{\left( {\log \varphi \left( N \right)} \right)}^\lambda }} \right)  \\
			\label{gujie23}	&  \lesssim \exp \left( { - {C_4}{{\left( {\log N} \right)}^\lambda }} \right) ,\quad N \gg 1,
		\end{align}
		where \eqref{gujie22} follows from the asymptotic relation\footnote{This, as well as the subsequent asymptotic relation, can be established via L'H\^opital's rule; the details are therefore omitted.}
		\[\int_X^{ + \infty } {{r^{d - 1}}\exp \left( { - {{\left( {\log r} \right)}^\lambda }} \right)dr}  \sim {\lambda ^{ - 1}} X^d {\left( {\log X} \right)^{ - \lambda  + 1}}\exp \left( { - {{\left( {\log X} \right)}^\lambda }} \right),\quad X \gg 1,\]
		and \eqref{gujie23} is due to \eqref{gujie222}.
		
		Finally, combining \eqref{o11} and \eqref{gujie23} with \eqref{J1J2}, we conclude that as $ N \to +\infty $:
		\[\mathbf{Error}_N^{\mathrm{dis}}\left(w,f,\rho,\theta\right)   \lesssim \exp \left( { - c{{\left( {\log N} \right)}^\lambda }} \right) ,\quad\forall \theta \in \mathbb{T}^d,\]
		for some constant $ 0 < c < \min \left\{ {{C_3},{C_4}} \right\} $. This completes the proof of Case \ref{item3:roman_III}.
		\vspace{3mm}

		Turning to Case \ref{item3:roman_IV}, we have	$ \widetilde \Delta \left( x \right) \sim \exp \left( {{x^\alpha }} \right) $ with $   \alpha  >0 $. For brevity, we retain the nonresonance condition and notation from Case \ref{item3:roman_III}. We quantitatively estimate $\mathcal{J}_1(N, \theta)$ and $\mathcal{J}_2(N, \theta)$ by choosing  the function $ \varphi \left( N \right) $ for $ N\gg 1 $ as
		\begin{equation*} 
			\varphi \left( N \right):= {N^{\left( {\alpha \beta  + d } \right)^{-1}}}{{\left( {\log N} \right)}^{ - 2\left( {\alpha \beta  + d } \right)^{-1}}}.
		\end{equation*}	
		For any $ 0 < \xi  < \alpha {\left( {\alpha \beta  + d} \right)^{ - 1}} $,  similar arguments yield
		\[{\mathcal{J}_1}\left( {N,\theta } \right) \lesssim \exp \left( { - {C_3}{{\left( {\frac{N}{{\Delta  \circ \varphi \left( N \right)}}} \right)}^{{\beta ^{ - 1}}}}} \right) \ll \exp \left( { - {N^\xi }} \right),\quad N \gg 1,\]
		and
		\[
		\mathcal J_2(N,\theta)\lesssim
		\exp\left(-\tfrac12\varphi(N)^\alpha\right)
		\ll\exp(-N^\xi),\quad N\gg1.
		\]
		Here we have used the asymptotic relation
		\[\int_X^{ + \infty } {{r^{d - 1}}\exp \left( { - {r^\alpha }} \right) {d}r}  \sim {\alpha ^{ - 1}}{X^{d - \alpha }}\exp \left( { - {X^\alpha }} \right), \quad  X \gg 1.\]
		Consequently, we arrive at the desired conclusion as $ N \to +\infty $:
		\[\mathbf{Error}_N^{\mathrm{dis}}\left(w,f,\rho,\theta\right)   \ll \exp \left( { - {N^\xi }} \right) ,\quad\forall \theta \in \mathbb{T}^d.\]
		As a remark, while the convergence rate could be further refined (e.g., in the spirit of \cite{TL25b}), the index $\xi$ cannot be improved using the current approach. This completes the proof of Case \ref{item3:roman_IV}.\vspace{3mm}
		
		We now complete the proofs of all cases in Theorem \ref{OPTT2}.

	\end{proof}

	\section*{Acknowledgements} 
	The authors express their sincere gratitude to Professor Bassam Fayad for his valuable comments on an earlier version of this manuscript. In particular, his insights regarding Case \ref{item:roman_II} of Theorem \ref{OPTT1} significantly improved the quality of this work. Z. Tong  was supported by the National Natural Science Foundation of China (Grant No. 12601347) and the China Postdoctoral Science Foundation (Grant No. 2025M783102). Y. Li was supported in part by the National Natural Science Foundation of China (Grant Nos. 12071175, 12471183 and 12531009). ChatGPT was used exclusively to enhance the language and readability of the manuscript.

\end{document}